\documentclass{amsart}

\input{xypic.tex} \xyoption{all}
\usepackage{amsthm,amssymb,latexsym,amsmath}
\usepackage[all]{xypic}

\newcommand{\aaa}{\mathbf{A}}
\newcommand{\aat}{A}
\newcommand{\aatp}{A_P}

\newcommand{\bb}{\mathbf{B}}

\newcommand{\bbt}{B}
\newcommand{\bbtp}{B_P}

\newcommand{\cala}{{\mathcal A}}
\newcommand{\calb}{{\mathcal B}}
\newcommand{\calc}{{\mathcal C}}
\newcommand{\calf}{{\mathcal F}}
\newcommand{\calh}{{\mathcal H}}
\newcommand{\calm}{{\mathcal M}}
\newcommand{\calmdgr}{{\mathcal M}^{\rm gr}_d}

\newcommand{\calms}{{\mathcal M}^{\rm st}}
\newcommand{\calp}{{\mathcal P}}
\newcommand{\calt}{{\mathcal T}}
\newcommand{\calv}{{\mathcal V}}
\newcommand{\calvdgr}{{\mathcal V}^{\rm gr}_d}

\newcommand{\calz}{{\mathcal Z}}
\newcommand{\catA}[1]{{\mathfrak A}}
\newcommand{\catI}[1]{{\mathfrak I}}
\newcommand{\catS}[1]{{\mathfrak S}}
\newcommand{\cpx}{\mathbb{C}}

\newcommand{\F}{{\mathbb F}}

\newcommand{\id}{{\mathbf 1}}
\newcommand{\iit}{I}
\newcommand{\iitp}{I_P}

\newcommand{\oh}{{\mathcal O}}

\newcommand{\opn}{{\mathcal O}_{\mathbb{P}^n}}

\newcommand{\ox}{{\mathcal O}_{\mathbb{X}}}

\newcommand{\p}[1]{{\mathbb{P}^{#1}}}
\newcommand{\pd}{{\mathbb{P}^d}}
\newcommand{\pn}{{\mathbb{P}^n}}
\newcommand{\pum}{\mathbb{P}^1}

\newcommand{\simto}{\stackrel{\sim}{\to}}

\newcommand{\tW}{\overline{W}}

\newcommand{\xx}{\mathbb{X}}
\newcommand{\xxt}{X}

\newcommand{\yy}{\mathbb{Y}}

\newcommand{\Z}{\mathbb{Z}}

\DeclareMathOperator{\codim}{{codim}}

\DeclareMathOperator{\Hom}{Hom}
\DeclareMathOperator{\ho}{H}
\DeclareMathOperator{\spec}{Spec }

\newtheorem{theorem}{Theorem}[section]

\newtheorem{proposition}[theorem]{Proposition}
\newtheorem{lemma}[theorem]{Lemma}
\newtheorem{corollary}[theorem]{Corollary}

\newtheorem{remark}[theorem]{Remark}
\newtheorem{example}[theorem]{Example}
\newtheorem{definition}[theorem]{{\bf Definition}}

\begin{document}

\title{ADHM construction of perverse instanton sheaves}

\author{Abdelmoubine Amar Henni}
\author{Marcos Jardim}
\author{Renato Vidal Martins}

\address{IMECC - UNICAMP \\
Departamento de Matem\'atica \\ Caixa Postal 6065 \\
13083-970 Campinas-SP, Brazil}
\email{jardim@ime.unicamp.br}
\address{ICEx - UFMG \\
Departamento de Matem\'atica \\
Av. Ant\^onio Carlos 6627 \\
30123-970 Belo Horizonte MG, Brazil}
\email{renato@mat.ufmg.br}
\address{IMECC - UNICAMP \\
Departamento de Matem\'atica \\
13083-970 Campinas-SP, Brazil}
\email{henni@ime.unicamp.br}
\maketitle

\begin{abstract}
We present a construction of framed torsion free instanton sheaves on a projective variety containing a fixed line which further generalizes the one on projective spaces. This is done by generalizing the so called ADHM variety. We show that the moduli space of such objects is a quasi projective variety, which is fine in the case of projective spaces. We also give an ADHM categorical description of perverse instanton sheaves in the general case, along with a hypercohomological characterization of these sheaves in the particular case of projective spaces.
\end{abstract}


\section{Introduction}

An unexpected connection between theoretical physics and algebraic geometry appeared in the late 1970's, when Atiyah, Drinfeld, Hitchin and Manin provided a complete classification of instantons on the 4-dimensional sphere $S^4$ using algebraic geometric techniques. More precisely, these authors used the Penrose-Ward correspondence between instantons on $S^4$ and certain holomorphic vector bundles on $\p3$ together with a characterization of holomorphic vector bundles on $\p3$ due to Horrocks \cite{ADHM}.

Nowadays, such link between theoretical physics and algebraic geometry can be found many forms, perhaps the most prominent of which is the so-called Hitchin-Kobayashi correspondence.

Later, Donaldson noticed in \cite{D1} that (framed) instantons on $S^4$ were also in correspondence with (framed) holomorphic bundles on $\p2$, while Nakajima considered in \cite{N2} framed torsion-free sheaves on $\p2$.

On a different direction, Mamone Capria and Salamon \cite{MCS} generalized the Penrose-Ward correspondence to a correspondence between quaternionic instantons on $\mathbb{H}\mathbb{P}^k$ and certain holomorphic
vector bundles on $\mathbb{P}^{2k+1}$. This paper motivated Okonek and Spindler to introduce the notion of \emph{mathematical instanton bundles} on $\mathbb{P}^{2k+1}$ \cite{OS}. Since then, such objects have attracted the attention of many authors, see for instance \cite{AO,CTT,ST} and the references therein.

More recently, the following generalization of mathematical instanton bundles was proposed in \cite{J-i}: an {\em instanton sheaf} on $\pn$ ($n\geq2$) is a torsion free sheaf $E$ which is the cohomology of a \emph{linear monad} of the form:
\begin{equation} \label{eqone}
\opn(-1)^{\oplus c} \longrightarrow \opn^{\oplus a} \longrightarrow\opn(-1)^{\oplus c}.
\end{equation}
When so, $c=h^1(E(-1))$ and is called the {\em charge} of $E$, while $a=2c+r$ where $r$ is its rank. In general, a sheaf $E$ on $\pn$ is said of {\em trivial splitting type} if there is an isomorphism $\phi:E|_\ell\to{\mathcal O}_{\ell}^{\oplus r}$ for some line $\ell\subset\pn$, and, if so, the pair $(E,\phi)$ is called a {\em framed sheaf}. With these definitions in mind, a mathematical instanton bundle in the sense of \cite{AO,OS} is a rank $2k$ locally free instanton sheaf on $\mathbb{P}^{2k+1}$ of trivial splitting type. Moreover, the framed torsion free sheaves considered by Nakajima in his extension of Donaldson's classification of framed holomorphic vector bundles \cite[Ch. 2]{N2} are precisely the framed instanton sheaves on $\p2$. Similarly to what was done in these two cases, cohomological characterizations of torsion free instanton sheaves on $\pn$ can be found, for instance, in \cite{CM,J-i}.

The concept of framed instanton sheaf readily generalizes from a projective space to a projective variety $\yy$ containing a line $\ell\subset\yy$, simply by substituting $\yy$ for $\pn$ in the monad above. This idea has already been explored, for instance, in L. Costa and R. M. Mir\'o-Roig's \cite{CM}. The goal of this paper is to present a construction of framed torsion free instanton sheaves on $\yy$ which further generalizes the construction of framed torsion free sheaves on $\p2$ done in \cite{D1,N2} and on $\p3$ \cite{FJ2};  these are in turn a generalization of the original ADHM construction of instantons \cite{ADHM,D1}. A research announcement outlining the case $\yy=\pn$ appeared in \cite{J-cr}.

In this way, we provide an explicit parametrization of the moduli space of framed instanton sheaves on $\yy$ via matrices satisfying certain quadratic equations. The first step is generalizing ADHM data and even the very ADHM equation, which we do in, respectively, Sections \ref{subdat} and \ref{subadv}. Afterwards, we build the {\em quotient ADHM variety} $\mathcal{M}^{\rm st}_{\yy}$ {\em of stable points} in Section \ref{subqav}, as a natural candidate variety for the moduli problem. We prove in Theorem \ref{prpoto} that an open set within $\mathcal{M}^{\rm st}_{\yy}$, consisting of what we call {\em globally weak stable} points, is indeed in 1-1 correspondence to isomorphism classes of framed torsion free instanton sheaves on $\yy$ (with fixed rank and charge); hence the moduli of framed torsion free instanton sheaves in $\yy$ is a quasi projective variety since $\mathcal{M}^{\rm st}_{\yy}$ is so. A more general statement holds in the case $\yy=\pn$. In fact, we prove in Theorem \ref{thmfin} that the moduli space of framed instantons bundles on $\pn$ is fine.

Motivated by \cite{BFG,HL}, we also consider \emph{perverse instanton sheaves} on $\yy$ in Section \ref{perverse}, generalizing what was done by the second and third named authors in \cite{JVM2}. More precisely, A. Braverman, M. Finkelberg and D. Gaitsgory proved a result due to Drinfeld, \cite[Thm. 5.7]{BFG}, which extended Nakajima-Donaldson's discussion from torsion free to \emph{perverse} sheaves on $\p2$. Recently in \cite{HL}, M. Hauzer and H. Langer discussed perverse instanton sheaves on $\p3$ generalizing the ADHM construction of framed torsion free instanton sheaves in \cite{FJ2}, together with a more refined study of stability. Here, we further generalize the results in these two papers.

Roughly speaking, perverse instanton sheaves are complexes of sheaves of the form (\ref{eqone}) regarded as objects in the core of particular t-structures on $D^b(\yy)$ constructed in Subsections \ref{kashiwara} and \ref{tilting}. The extension of the theory from torsion free to perverse instanton sheaves corresponds to disregard the ``globally weak stable" condition on ADHM data satisfying the ADHM equation. This is the content of Theorem \ref{thm111}, a connection between what we call the {\em ADHM category} and the one of perverse instanton sheaves. Again, in the very case where $\yy=\pn$ it is also possible to give a hypercohomological characterization of perverse instanton sheaves which generalizes the well known cohomological characterization of torsion free instanton sheaves. This is Theorem \ref{thmfim}, which closes the paper.

\bigskip

\paragraph{\bf Acknowledgments.}
The first named author is supported by the FAPESP post-doctoral grant number 2009/12576-9. The second named author is partially supported by the CNPq grant number 302477/2010-1 and the FAPESP grant number 2005/04558-0. The third named author is partially supported by the CNPq grant number 304919/2009-8.


\section{Generalized ADHM Data}
\label{secdat}

For the remainder, $\yy$ is a projective scheme in $\pn=\mathbb{P}^n_\cpx$ which contains a line $\ell$. Fix homogeneous coordinates $(z_0:\ldots :z_d:x:y)\in\pn$ where $d=n-2$ such that $\ell$ is given by the equations $z_0=\ldots=z_d=0$. Set
$$
\ho_\yy:=\langle z_0,\ldots,z_d\rangle\subset\ho^0(\oh_\yy(1)).
$$


\subsection{The ADHM Data}
\label{subdat}

Let $V$ and $W$ be complex vector spaces of dimension, respectively, $c$ and $r$. Set
$$
\aaa:={\rm End}(V)^{\oplus 2}\oplus{\rm Hom}(W,V)\ \ \ \ \ \ \ \ \ \aaa':={\rm End}(V)^{\oplus 2}\oplus {\rm Hom}(V,W)
$$
$$
\bb=\bb(W,V)=\bb(r,c):=\aaa\times\aaa'
$$
and consider the affine spaces
$$
\aaa_\yy:=\aaa\otimes \ho_\yy\ \ \ \ \ \ \ \ \ \aaa'_\yy:=\aaa'\otimes \ho_\yy
$$
$$
\bb_{\yy}=\bb_{\yy}(W,V)=\bb_{\yy}(r,c):=\aaa_\yy\times\aaa'_\yy
$$

A point of $\mathbf{B}_\yy$ will be called in this paper an \emph{ADHM datum over $\yy$}. The ADHM data over $\yy=\pn$ were considered in \cite{J-cr} in a slightly different way; it will be the most relevant case in the present paper. It is important to point out that the subspace $\langle x,y\rangle\subset\ho^0(\mathcal{O}_\yy(1))$ will play no special role until Subsection \ref{subadv}.

One can write a point of $X\in\mathbf{B}_\yy$ as
$$
X=(Y,Y')
$$
with
$$
Y=(A,B,I)\ \ \ \ \ \ \ \ \ Y'=(A',B',J)
$$
where the above components are
$$
A = A_{0}\otimes z_0 + \cdots + A_{d}\otimes z_d\ \ \ \ \ \ \ \
A' = A_{0}'\otimes z_0 + \cdots + A_{d}'\otimes z_d
$$
$$
B = B_{0}\otimes z_0 + \cdots + B_{d}\otimes z_d\ \ \ \ \ \ \ \
B' = B_{0}'\otimes z_0 + \cdots + B_{d}'\otimes z_d
$$
$$
I = I_0\otimes z_0 + \cdots + I_d\otimes z_d\ \ \ \ \ \ \ \
J = J_0\otimes z_0 + \cdots + J_d\otimes z_d\
$$

\

\noindent with $A_{k}, B_k, A_{k}', B_k' \in {\rm End}(V)$, $I_k \in {\rm Hom}(W,V)$ and $J_k \in {\rm Hom}(V,W)$. Hence we naturally regard $A,B,A',B'\in {\rm Hom}(V,V\otimes \ho_\yy)$, and also $I\in {\rm Hom}(W,V\otimes \ho_\yy)$ and $J\in {\rm Hom}(V,W\otimes \ho_\yy)$. Setting $Y_k:=(A_k,B_k,I_k)\in\aaa$ and $Y_k':=(A_k',B_k',J_k)\in\aaa'$, sometimes it is convenient to write
$$
Y=Y_0\otimes z_0+\ldots+Y_d\otimes z_d=(Y_0,\ldots,Y_d)\in \aaa^{d+1}
$$
$$
Y'=Y_0'\otimes z_0+\ldots+Y_d'\otimes z_d=(Y_0',\ldots,Y_d')\in (\aaa')^{d+1}
$$
and consider the ADHM datum as
$$
X=(Y,Y')\in\aaa^{d+1}\times(\aaa')^{d+1}.
$$
For any $P\in\yy$ we define the \emph{evaluation maps} given on generators by
\begin{gather*}
\begin{matrix}
{\rm ev}_{P}^1: &\aaa_\yy & \longrightarrow & \mathbb{P}(\aaa)\\
               & Y_i\otimes z_i              & \longmapsto     & [z_i(P)Y_i]
\end{matrix}
\ \ \ \ \ \ \ \
\begin{matrix}
{\rm ev}_P^2: &\aaa'_\yy & \longrightarrow & \mathbb{P}(\aaa')\\
               & Y_i'\otimes z_i              & \longmapsto     & [z_i(P)Y_i'].
\end{matrix}
\end{gather*}
Note that $z_i(P)\in\cpx$ depends on a choice of trivialization of $\oh_\yy(1)$ at $P$ but the class on projective space does not. We set $Y_P:={\rm ev}_{P}^{1}(Y)$ and, similarly, $Y_P':={\rm ev}_{P}^{2}(Y')$. In particular, $A_P$, $B_P$, $A'_P$, $B'_P$, $I_P$ and $J_P$ are defined as well. For any subspace $S\subset V$, we are able to naturally well define the subspaces $A_P(S),B_P(S),A_P'(S),B_P'(S),I_P(W)$ and $\ker J_P$ of $V$. We also consider
\begin{gather*}
\begin{matrix}
{\rm ev}_{P}: &\mathbf{B}_{\yy} & \longrightarrow & \mathbb{P}(\aaa)\times\mathbb{P}(\aaa')\\
               & (Y, Y')              & \longmapsto     & (Y_P, Y_P')
\end{matrix}
\end{gather*}
and set $X_P:={\rm ev}_{P}(X)$. With this in mind we define the following.

\begin{definition}
Let $Y=(A,B,I)\in\aaa_\yy$ and $Y'=(A',B',J)\in\aaa'_\yy$. Let also $P$ be a point in $\yy$.
\begin{enumerate}
\item[(i)]$Y_P$ is said {\em stable} if there is no proper subspace $S\subset V$ for which hold the inclusions  $A_P(S),B_P(S),I_P(W)\subset S$;
\item[(ii)] $Y'_P$ is said {\em costable} if there is no nonzero subspace $S\subset V$ for which hold the inclusions $A_P'(S),B_P'(S)\subset S\subset\ker J_P$;
\item[(iii)]$Y_P$ is said {\em weak stable} if there is no subspace $S\subset V$ of codimension $1$ for which hold the inclusions  $A_P(S),B_P(S),I_P(W)\subset S$;
\item[(iv)] $Y'_P$ is said {\em weak costable} if there is no subspace $S\subset V$ of dimension $1$ for which hold the inclusions $A_P'(S),B_P'(S)\subset S\subset\ker J_P$;
\item[(v)] $Y$ is said {\em stable} if there is no proper subspace $S\subset V$ for which hold the inclusions  $A(S),B(S),I(W)\subset S\otimes \ho_\yy$;
\item[(vi)] $Y'$ is said {\em costable} if there is no nonzero subspace $S\subset V$ for which hold the inclusions $A'(S),B'(S)\subset S\otimes\ho_\yy$ and $S\subset \ker J$;
\item[(vii)] $Y$ is said {\em locally} (resp. {\rm globally}) {\rm stable} (corresp. {\rm weak stable}) if $Y_P$ is stable (corresp. weak stable) for some (resp.  every) $P\in \yy$;
\item[(viii)] $Y'$ is said {\em locally} (resp. {\rm globally}) {\rm costable} (corresp. {\rm weak costable}) if $Y_P$ is costable (corresp. weak costable) for some (resp.  every) $P\in \yy$.
\end{enumerate}
\end{definition}

We are now able to introduce the key definitions of this paper.

\begin{definition}
The datum $X=(Y,Y')\in\mathbf{B}_\yy$ is said
\begin{enumerate}
\item[(i)] {\rm stable} (resp. {\rm locally stable, locally weak stable, globally stable, globally weak stable}) if $Y$ is stable (resp. locally stable, locally weak stable, globally stable, globally weak stable);
\item[(ii)] {\em costable} (resp. {\rm locally costable, locally weak costable, globally costable, globally weak costable}) if $Y'$ is costable (resp. locally costable, locally weak costable, globally costable, globally weak costable);
\item[(iii)] {\em regular} (resp. {\rm locally regular, locally weak regular, globally regular, globally weak regular}) if it is both stable and costable (resp. locally stable and locally costable, locally weak stable and locally weak costable, globally stable and globally costable, globally weak stable and globally weak costable).
\end{enumerate}
\end{definition}

\begin{remark} \rm
For $\yy=\mathbb{P}^2$, i.e., $\mathbf{B}_\yy=\mathbf{B}$, stability, costability and regularity essentialy coincide with the usual notions for ADHM data (cf. \cite[Thm. 2.1]{N2}). For $\yy=\pn$, the present notions of global stability and global regularity correspond, respectively, to stability and regularity in \cite[p. 29]{FJ2} and \cite[Def. 2.1]{J-cr}; global stability along with regularity here correspond to semiregularity in \cite{FJ2,J-cr}; and the present notions of stability, costability and regularity have no parallel in \cite{FJ2,J-cr}.
\end{remark}

\begin{definition}
\label{defsig}
Let $Y=(A,B,I)\in\aaa_\yy$ and $P\in\yy$. The \emph{stabilizing subspace} $S_{Y_P}$ is the intersection of all subspaces $S\subset V$ for which hold the inclusions $A_P(S),B_P(S),I_P(W)\subset S$. The \emph{stabilizing subspace} $S_Y$ is the intersection of all subspaces $S\subset V$ such that $A(S),B(S),I(W)\subset S\otimes\ho_\yy$.
\end{definition}

If $S\subseteq V$ satisfies $A(S),B(S),I(W)\subset S\otimes\ho_\yy$, then one may consider
$$
Y|_S:=(A|_S,B|_S,I)\in\aaa_\yy(W,S).
$$
It is clear that $Y|_{S_{Y}}$ is stable and this justifies the term we use. A similar statement holds for points, that is, if $Y_P|_S:={\rm ev}_P^1(Y|_S)$ then $Y_{P}|_{S_{Y_P}}$ is stable as well. Moreover, $Y$ is stable if and only if $S_{Y}=V$ and $Y_{P}$ is stable if and only if $S_{Y_{P}}=V$.

\begin{proposition}
\label{prpst1}
$S_{Y_{P}}\subset S_{Y}$ for every $P\in\yy$. In particular, if $Y$ is locally stable then it is stable.
\end{proposition}

\begin{proof}
Let $S\subset V$ and consider $\yy\subset\pn$. Then $A(S)\subset S\otimes\ho_\yy$ if and only if $A_i(S)\subset S$ for $i=0,\ldots,d$, which holds if and only if  $\sum_{i=0}^{d}p_iA_i(S)\subset S$ for every $P=(p_0:\dots :p_d:a:b)\in\pn$. So applying the same to $B$ and $I$ we see that $A(S),B(S),I(W)\subset S\otimes\ho_\yy$ if and only if  $A_P(S),B_P(S),I_P(W)\subset S$ for every $P\in\pn$. The result follows.
\end{proof}

In \cite[Sec. 6.1]{HL} there is an example, with $c=2$, $r=1$ and $\yy=\mathbb{P}^3$, for which the converse of the above proposition does not hold, that is, there exists a stable $Y\in\aaa_{\mathbb{P}^3}$ such that $Y_{P}$ is not stable for every $P\in\mathbb{P}^3$. Equivalently, there exists $Y\in\aaa_{\mathbb{P}^3}$ such that $S_{Y_{P}}\subsetneqq S_{Y}=V$ for every $P\in\mathbb{P}^3$. So one may ask if, in general, at least $T_{Y}:=\sum_{P\in\yy} S_{Y_{P}}= S_{Y}$. If $c=1$ and $\yy$ is reduced, this trivially holds. If $c=2$, as in the example of \cite{HL}, and $\yy$ is integral, this holds as well. In fact, write $Y=(A,B,I)$ and note that for every $P\in\yy$ hold: $\iitp(W)\subset S_{Y_{P}}$ and, also,  $\iitp=0$ if and only if $S_{Y_{P}}=0$. Besides, $S_{Y}=0$ if and only if $\iit=0$. But the entries of $\iit$ are sections in $\ho^0(\oh_\yy(1))$ and $\yy$ is reduced, thus $S_{Y}=0$ if and only if $\iitp =0$ for every $P\in\yy$. Set $L_{Y}:=\sum_{P\in\yy}\iit_P(W)$. If $L_{Y}=0$ then $T_{Y}=S_{Y}=0$; if $L_{Y}=V$ then $T_{Y}=S_{Y}=V$; and if $\dim(L_{Y})=1$ then either exists $P\in \yy$ such that $\iitp(W)$ is not invariant by $\aatp$ or $\bbtp$ and then $S_{Y_{P}}=V$ which implies $T_{Y}=S_{Y}=V$, or, otherwise, $T_{Y}=S_{Y}=L_{Y}$ since $\yy$ is irreducible.

The above assumption on $\yy$ to be reduced for $c=1$ is necessary. In fact, take $\yy$ to be the nonreduced variety given by $z_0^2=0$ in $\mathbb{P}^2={\rm Proj}\, \cpx[z_0,x,y]$. Then $z_0\in \ho^0(\oh_\yy(1))$ is not the zero section but vanishes at every point of $\yy$. So if $r=1$ as well and $Y$ is such that $\iit=(z_0)$ then $S_{Y}$ cannot be zero. On the other hand, $\iitp=0$ for every $P\in\yy$, so $S_{Y_{P}}=0$ and hence $T_{Y}=0$.

The variety $\yy$ needs also to be irreducible for the case $c=2$. In fact, take $\yy$ to be the reducible variety given by $z_0z_1=0$ in $\mathbb{P}^3={\rm Proj}\, \cpx[z_0,z_1,x,y]$. Consider $Y=(\aat,\bbt,\iit)\in\aaa_{\yy}(1,2)$ where
$$
\aat =
\left( \begin{array}{cc}
z_0 &  z_1         \\
0 & z_0          \\
\end{array} \right)
\ \ \ \ \ \ \ \
I =
\left( \begin{array}{c}
z_0    \\
z_0
\end{array} \right)
$$
and $\bbt=0$. Let $P=(p_0:p_1:a:b)\in\yy$. If $p_0\neq 0$ then $\iitp(W)=N:=\langle (1,1)\rangle$ and $N$ is also $\aatp$-invariant because $p_1=0$; hence $S_{Y_{P}}=N$ in such a case. If $p_0=0$ then $S_{Y_{P}}=0$. Hence $T_{Y}=N$. On the other hand, if $N=S_{Y}$ then, in particular, $\aat(N)\subset N\otimes \ho_\yy$; but
$$
\aat(N)=\langle(z_0+z_1,z_0)\rangle\not\subset \langle (z_0,z_0),(z_1,z_1)\rangle=N\otimes \ho_\yy
$$
which implies that  $T_{Y}\subsetneqq S_{Y}$.

For $c=3$, we can have $T_{Y}\subsetneqq S_{Y}$ even when $\yy=\mathbb{P}^3={\rm Proj}\, \cpx[z_0,z_1,x,y]$. In fact, consider $Y=(\aat,\bbt,\iit)\in\aaa_{\mathbb{P}^3}(1,3)$ where
$$
\aat = \left( \begin{array}{ccc}
      z_1 &  z_1 &  0         \\
       0        &  z_1 &  z_0 \\
       0        &  0        &  2z_1
      \end{array} \right)
\ \ \ \ \ \ \ \
I = \left( \begin{array}{c}
                z_0   \\
                z_0 \\
              z_1
                \end{array} \right)
$$
and $\bbt=0$. Use \cite[Lem. 3.2.(i)]{JVM2} to get $S_{Y_{P}}=\sum_{i=0}^2 \aatp^i\iitp(W)$ and compute
$$
\aat\iit = \left( \begin{array}{c}
                2z_0z_1  \\
                2z_0z_1   \\
                2z_1^2
                \end{array} \right)
\ \ \ \ \ \ \ \
(\aat)^2\iit = \left( \begin{array}{c}
                4z_0z_1^2  \\
                4z_0z_1^2   \\
                4z_1^3
                \end{array} \right).
$$
Let $P=(p_0:p_1:a:b)\in\mathbb{P}^3$. If $p_0,p_1\neq 0$ then $S_{Y_{P}}=N:=\langle (1,1,0),(0,0,1)\rangle$.  If $p_0=0$  (resp. $p_1=0$) then $S_{Y_{P}}=\langle (0,0,1)\rangle$ (resp. $S_{Y_{P}}=\langle (1,1,0)\rangle$). Thus $T_{Y}=N$.  However, if $N=S_{Y}$ then, as seen above, $\aat(N)\subset N\otimes \ho_\yy$; but
\begin{align*}
\aat(N) &=\langle(2z_1,z_1,0),(0,z_0,2z_1)\rangle\\
& \not\subset \langle (z_0,z_0,0),(0,0,z_0),(z_1,z_1,0),(0,0,z_1)\rangle\\
&=N\otimes \ho_\yy
\end{align*}
which implies that  $T_{Y}\subsetneqq S_{Y}$.

We define $\mathbf{B}_\yy^{\rm st}$, $\mathbf{B}_\yy^{\rm lws}$, $\mathbf{B}_\yy^{\rm ls}$, $\mathbf{B}_\yy^{\rm gws}$, $\mathbf{B}_\yy^{\rm gs}$,  $\mathbf{B}_\yy^{\rm gwr}$ and $\mathbf{B}_\yy^{\rm gr}$  as the subsets of $\mathbf{B}_\yy$ consisting of stable, locally weak stable, locally stable, globally weak stable, globally stable, globally weak regular and globally regular ADHM data over $\yy$, respectively. Clearly, each of these sets are open subsets of $\mathbf{B}_\yy$ (in the Zariski topology), and one has strict inclusions
$$
\begin{matrix}
\mathbf{B}_\yy^{\rm gr} & \subset & \mathbf{B}_\yy^{\rm gs} & \subset & \mathbf{B}_\yy^{\rm ls} & &\\
\cap &  & \cap &  & \cap & &\\
\mathbf{B}_\yy^{\rm gwr} & \subset & \mathbf{B}_\yy^{\rm gws} & \subset & \mathbf{B}_\yy^{\rm lws} & \subset & \mathbf{B}_\yy^{\rm st}
\end{matrix}
$$


\subsection{The ADHM Variety}
\label{subadv}

One of the main goals of this paper is to consider, for data $X=((A,B,I),(A',B',J))\in\bb_\yy$, the \emph{generalized ADHM equation}
\begin{equation}
\label{c4}
AB'-BA'+IJ+(B'-B)\otimes x+(A-A')\otimes y=0
\end{equation}
which we also call the \emph{ADHM equation over} $\yy$.

Considering the map
\begin{gather*}
\begin{matrix}
\mu  : & \mathbf{B}_\yy & \longrightarrow &{\rm End}(V)\otimes \ho^0({\mathcal O}_\yy(2))   \\
                             & X & \longmapsto     & AB'-BA'+IJ+(B'-B)\otimes x+(A-A')\otimes y
\end{matrix}
\end{gather*}
we set the space of all solutions to the ADHM equation over $\yy$ as
$$
\calv_\yy=\calv_\yy(W,V)=\calv_\yy(r,c):=\mu^{-1}(0)
$$
which is an affine variety. We call it the \emph{ADHM variety over} $\yy$ of which we select the subvarieties
$$
\begin{matrix}
\calv_\yy^{\rm gr} & \subset & \calv_\yy^{\rm gs} & \subset & \calv_\yy^{\rm ls} & &\\
\cap &  & \cap &  & \cap & &\\
\calv_\yy^{\rm gwr} & \subset & \calv_\yy^{\rm gws} & \subset & \calv_\yy^{\rm lws} & \subset & \calv_\yy^{\rm st}
\end{matrix}
$$
consisting of globally regular, globally stable, locally stable, globally weak regular, globally weak stable, locally weak stable and stable points of $\calv_\yy$.

\begin{example}
\label{exappn}
\emph{The most relevant case is $\yy=\pn$. When so, note that if $X=((A,B,J),(A',B',I))\in\calv_\pn$ then necessarily $A=A'$ and $B=B'$. So one may write the datum as $X=(A,B,I,J)$ and the ADHM equation as
$$
[A,B]+IJ=0
$$
as extensively done in the literature. For instance, note that, for $\yy=\mathbb{P}^2$, (\ref{c4}) reduces to the most usual one taken in ${\rm End}(V)$. The case where $\yy=\mathbb{P}^3$ was considered in \cite{D1,FJ2} in the context of Yang-Mills theory and the Penrose correspondence. It can be written as follows:
$$
[ A_{0} , B_{0} ] + I_0J_0 = 0
$$
$$
[ A_{1} , B_{1} ] + I_1J_1 = 0
$$
$$
[ A_{0} , B_{1} ] + [ B_{0} , A_{1} ] + I_0J_1 + I_1J_0 = 0
$$
which were called {\em complex ADHM equations} in \cite[Eqs. 7-9]{FJ2}. More generally, in $\yy=\pn=\mathbb{P}^{d+2}$, (\ref{c4}) can be broken down into ${d+2}\choose{2}$ equations involving the linear maps $A_k$, $B_{k}$, $I_k$ and $J_k$:
$$
[A_{k}, B_{k} ]+I_kJ_k=0 \ \  \ \ \ \ \ \ \ \ \ \ \ \ \ \ \ \ \ \  \ \ \ \ \ \ k=0,\dots, d
$$
$$
[ A_{k} , B_{m} ]+ [ B_{k} , A_{m} ] + I_kJ_m + I_mJ_k = 0  \ \ \ \ \ \  k<m=0,\dots,d
$$
to which we refer here as the $d$-{\em dimensional ADHM equations}.}
\end{example}

\begin{example}
\label{exepxp}
\emph{Consider the case where $\yy={\rm S}\subset\mathbb{P}^3={\rm Proj}\, \cpx[z_0,z_1,x,y]$ is the scroll given by the equation
$$
z_0y-z_1x=0
$$
which is isomorphic to $\pum\times\pum$. It is easily found data in ${\rm S}$ which do satisfy the ADHM equation though not satisfying the conditions $A=A'$ and $B=B'$. In fact, let $X=((A,B,I),(A',B',J))\in\bb_{\rm S}$ be such that
$$
A=C\otimes z_0\ \ \ \ \ \ \ \ \ \ \ \ B=C\otimes z_1\ \ \ \ \ \ \ \ \ \ A'=B'=J=0
$$
where $C$ is a $c\times c$ matrix. Then
\begin{align*}
AB'-BA'+IJ+(B'-B)\otimes x+(A-A')\otimes y &=(-C\otimes z_1)\otimes x+(C\otimes z_0)\otimes y \\
                                           &=(z_0y-z_1x)C=0.
\end{align*}
The coordinates ADHM equations in this case are:
$$
A_{0}  B_{0}'-B_0A_0'+ I_0J_0 = 0
$$
$$
A_{1} B_{1}'-B_1A_1'+ I_1J_1 = 0
$$
$$
A_{0} B_{1}'- B_{0}A_{1}'+A_{1} B_{0}'- B_{1}A_{0}' + I_0J_1 + I_1J_0 = 0
$$
$$
A_{1}=A_1'\ \ \ \ \ \ \ \  B_0=B_0'\ \ \ \ \ \ \ \  A_0-A_0'=B_1-B_1'.
$$}
\end{example}

\begin{example}
\label{exequa}
\emph{Now consider the hyperquadric $\yy={\rm Q}\subset\mathbb{P}^4={\rm Proj}\, \cpx[z_0,z_1,z_2,x,y]$ given by the equation
$$
z_0y+z_1x+z_2^2=0.
$$
Let $X=((A,B,I),(A',B',J))\in\bb_{\rm Q}$ be such that
$$
A=C\otimes z_0\ \ \ \ \ B=-C\otimes z_1\ \ \ \ \ \ \ \ I=I'\otimes z_2\ \ \ \ \ \ \ A'=B'=0\ \ \ \ \ \ \ J=J'\otimes z_2
$$
where $C,I',J'$ are, respectively, $c\times c$, $c\times r$, $r\times c$ matrices such that $C=I'J'$. Then we have
\begin{align*}
AB'-BA'+IJ+(B'-B)\otimes x+(A-A')\otimes y &=0 \\
(I'\otimes z_2)\otimes(J'\otimes z_2)+(C\otimes z_1)\otimes x+(C\otimes z_0)\otimes y &=0 \\
                                             z_2^2(I'J')+(z_1x+z_2y)C &=0 \\
                                             (z_0y+z_1x+z_2^2)C &=0
\end{align*}
and so here on finds data in the ADHM variety not satisfying the conditions $A=A'$, $B=B'$ either. The coordinates ADHM equations in this case are:
$$
A_{0}  B_{0}'-B_0A_0'+ I_0J_0 = 0
$$
$$
A_{1} B_{1}'-B_1A_1'+ I_1J_1 = 0
$$
$$
A_{0} B_{1}'- B_{0}A_{1}'+A_{1} B_{0}'- B_{1}A_{0}' + I_0J_1 + I_1J_0 = 0
$$
$$
A_{0} B_{2}'- B_{0}A_{2}'+A_{2} B_{0}'- B_{2}A_{0}' + I_0J_2 + I_2J_0 = 0
$$
$$
A_{1} B_{2}'- B_{1}A_{2}'+A_{2} B_{1}'- B_{2}A_{1}' + I_1J_2 + I_2J_1 = 0
$$
$$
A_{1}=A_1'\ \ \ \ \ \ \ A_{2}=A_2'\ \ \ \  \ \ \ B_0=B_0'\ \ \ \ \ \ B_2=B_2'
$$
$$
A_{2}  B_{2}'-B_2A_2'+ I_2J_2 =A_0-A_0'=B_1'-B_1
$$}
\end{example}


\subsection{The Quotient ADHM Variety of Stable Points}
\label{subqav}

The importance of the ADHM variety will be clear in the next sections and has to do with the very aim of this work, that is, the construction of framed instanton sheaves by means of ADHM data. In fact,  we will see how to build such sheaves from points of $\calv_\yy$; how to establish a correspondence between points of $\calv_\yy^{\rm gws}$ and torsion free instanton sheaves and, also, how to even get a moduli space of instanton bundles on $\pn$ from points of $\calv_\pn^{\rm gr}$. In order to do so, we introduce an action of $G:=GL(V)$ on $\mathbf{B}_\yy$.

Given $g\in G$, $Y=(A,B,I)\in\aaa$ and $Y'=(A',B',J)\in\aaa'$ one defines
\begin{align*}
g\cdot Y &:= (gAg^{-1},gBg^{-1},gI) \\
g\cdot Y'&:= (gA'g^{-1},gB'g^{-1},Jg^{-1}).
\end{align*}
The action naturally extends to $\mathbf{B}_\yy$ as
$$
g\cdot(Y\otimes z,Y'\otimes w)=((g\cdot Y)\otimes z,(g\cdot Y')\otimes w)
$$
for any $Y\in\aaa$, $Y'\in\aaa'$ and $z,w\in \ho_{\yy}$.

\begin{proposition}
\label{Sh1}
If $X\in\mathbf{B}_\yy$ is stable, then its stabilizer subgroup $G_{X}$ is trivial.
\end{proposition}
\begin{proof}
Write $X=(Y,Y')$ with $Y=(A,B,I)$. If $G_{X}$ is nontrivial, take $g\neq\mathbf{1}_V$ in $G_{X}$. Since $I\neq 0$ because $X$ is stable and also $gI=I$, then $S:=\ker (g-\mathbf{1}_V)$ is proper and $A(S),B(S),I(W)\subset S\otimes\ho_{\yy}$ which contradicts the stability of $X$. Thus $G_X$ is trivial.
\end{proof}

The action of $G$ on $\yy$ naturally restricts to $\calv_\yy$, and induces an action of $G$ on $\Gamma(\calv_\yy)$, the ring of regular functions of $\calv_\yy$. Fix $l>0$ and consider the group homomorphism $\chi:G\to\cpx^*$ given by $\chi(g)=(\det g)^l$. This can be used to lift the action of $G$ from $\calv_\yy$ to $\calv_\yy\times\cpx$ as follows
\begin{equation}
g\cdot (\xxt,z) := (g\cdot\xxt,\chi(g)^{-1}z)
\end{equation}
for any $\xxt\in\calv_\yy$ and $z\in\cpx$. Then one can form the variety
$$ \calv_\yy /\!/_{\chi}\,G :=
{\rm Proj}\left( \bigoplus\limits_{n\geq0} \Gamma(\calv_\yy)^{G,\chi^n} \right)$$
where
$$
\Gamma(\calv_\yy)^{G,\chi^n} := \left\{ f\in \Gamma(\calv_\yy) ~|~ f(g\cdot \xxt)= \chi(g)^n f(\xxt) ~\forall g\in G \right\}
$$
Clearly, $\calv_\yy /\!/_{\chi} G$ is projective over ${\rm Spec}\left(\Gamma(\calv_\yy)^{G}\right)$,
and it is quasi-projective over $\cpx$. The GIT tells us that $\calv_\yy /\!/_{\chi} G$ is
the space of orbits $G\cdot\xxt\subset\calv_\yy$ such that the lifted orbit
$G\cdot(\xxt,z)$ is closed within $\calv_\yy\times\cpx\setminus\{0\}$ for all $z\ne0$.


\begin{proposition}
\label{closed-stable-d}
The orbit $G\cdot(\xxt,z)$ is closed for $z\ne0$ if and only if
$\xxt$ is stable.
\end{proposition}

\begin{proof}
The usual proof of the case $\yy=\mathbb{P}^2$ generalizes to the current framework. Take $X=(Y,Y')$ with $Y=(\aat,\bbt,\iit)$ and, first, assume that the orbit $G\cdot(\xxt,z)$ is not closed. Then there is a nontrivial one parameter subgroup $\lambda:\cpx^*\to G$ such that the limit
\begin{equation}\label{lim1}
(L,w) = \lim_{t\to0} \lambda(t)\cdot (X,z)
\end{equation}
exists but does not belong to the orbit $G\cdot(X,z)$.

Take a weight decomposition of $V$ with respect to $\lambda$, so that $V = \oplus_m V(m)$. The existence of the limit implies that
$$
\aat (V(m)), \bbt (V(m)) \subset \left( \oplus_{n\ge m} V(n) \right)\otimes\ho_\yy,
$$
$$
\iit(W) \subset \left( \oplus_{n\ge 0} V(n) \right)\otimes\ho_\yy.
$$
Set $S=\oplus_{n\ge 0} V(n)$, so that $\aat(S),\bbt(S),\iit(W)\subset S\otimes \ho_\yy$. We claim that $S$ is a proper subspace of $V$, which implies that $X$ is not stable. Indeed, the existence of the limit (\ref{lim1}) implies that $\det\lambda(t)=t^N$
for some $N\leq0$. If $N=0$, then actually $\lambda(t)=\id_V$ and $V=V(0)$, which
contradicts the fact that the limit (\ref{lim1}) does not belong to the orbit
$G\cdot(X,z)$. Hence $N<0$, which implies that $S$ is proper, as desired.

Conversely, assume that $X$ is not stable. Then there exists a proper subspace $S\subset V$ such that $A(S),B(S),I(W) \subset S\otimes\ho_\yy$. Taking any subspace $T\subset V$ such that $V=S\oplus T$, the maps $\aat,\bbt$
and $\iit$ may be written, with respect to this decomposition, as follows
$$A,B = \left( \begin{array}{cc} \star & \star \\ 0 & \star \end{array} \right)
~~ {\rm and} ~~
I = \left( \begin{array}{c} \star \\ 0 \end{array} \right). $$
Defining the $1$-parameter subgroup $\lambda:\cpx^*\to G$ as
$$ \lambda(t) = \left( \begin{array}{cc} \id_S & 0 \\ 0 & t^{-1}\id_T \end{array} \right),$$
note that
$$ \lambda(t)A\lambda(t)^{-1},
\lambda(t)B\lambda(t)^{-1} =
\left( \begin{array}{cc} \star & t\cdot\star \\ 0 & \star \end{array} \right)
\ \ {\rm and} \ \
\lambda(t)I = I. $$
It follows that $L=\lim_{t\to0} \lambda(t)\cdot X$ exists. Thus
$$\lim_{t\to0} \lambda(t)\cdot (X,z) = (L,0),$$
which means that the orbit $G\cdot(X,z)$ is not closed within
$\calv_\yy\times\cpx\setminus\{0\}$.
\end{proof}

From the above proposition,  we are able to introduce a variety which will play a central role in Section \ref{secmod}, namely, the \emph{quotient ADHM variety of stable points over} $\yy$, defined as
$$ {\mathcal M}_\yy^{\rm st}={\mathcal M}_\yy^{\rm st}(r,c) :=\left. \left\{ \begin{array}{c} {\rm stable~solutions~of~the} \\ {\rm ADHM~equation}
\end{array} \right\} \right/ G\simeq \calv_\yy /\!/_{\chi} G. $$

So ${\mathcal M}_\yy^{\rm st}=\calv_\yy^{\rm st}/G $ is a quasiprojective variety. Note that we may consider the following sequence of varities
$$
\begin{matrix}
\calm_\yy^{\rm gr} & \subset & \calm_\yy^{\rm gs} & \subset & \calm_\yy^{\rm ls} & &\\
\cap &  & \cap &  & \cap & &\\
\calm_\yy^{\rm gwr} & \subset & \calm_\yy^{\rm gws} & \subset & \calm_\yy^{\rm lws} & \subset & \calm_\yy^{\rm st}
\end{matrix}
$$
consisting of globally regular, globally stable, locally stable, globally weak regular, globally weak stable, locally weak stable and stable orbits of $\calm_\yy^{\rm st}$.  The next natural questions are to determine whether $\calms_M$ is irreducible and nonsingular, and to compute its dimension . If $\calms_\yy$ is not irreducible, one would like to characterize and count its irreducible components. If $\calms_\yy$ is singular, one would like to characterize the singularity locus. We ask some of these questions for the case of projective spaces in Section \ref{secmod}. Notice that since $G$ acts freely and properly on $\calv^{\rm st}_\yy$, it follows that $\calms_\yy$ is irreducible/nonsingular if and only if $\calv^{\rm st}_\yy$ is irreducible/nonsingular, and that
$$ \dim \calms_\yy(r,c) = \dim \calv^{\rm st}_\yy(r,c) - c^2. $$

\begin{remark}\label{GIT}
\emph{Any stable solution of the ADHM equation is GIT stable by Propositions \ref{Sh1} and \ref{closed-stable-d} (see also \cite[Prp. 4.3]{HL}). Then ${\mathcal M}_\yy^{\rm st}$ is a good categorical quotient since the group $G$ is reductive \cite[Thm. 1.10]{mumford}; in particular $\calm_\yy^{\rm gws}$ is also a good categorical quotient.}
\end{remark}


\section{Torsion free instanton sheaves}
\label{secfra}

Our aim here is to establish a correspondence between torsion free instanton sheaves of trivial splitting type on $\yy$ and globally stable solutions of the ADHM equation over $\yy$.


\subsection{The ADHM construction}

In this subsection we will see how to construct coherent sheaves on $\yy$ which restrict trivially to $\ell$, out of ADHM data over $\yy$. To begin with, for any datum $X=((A,B,I),(A',B',J))\in\bb_\yy$, consider the sequence of sheaf maps
\begin{equation}
\label{monad.pn}
E^\bullet_{X} ~:~ V\otimes\oh_\yy(-1) \stackrel{\alpha}{\longrightarrow} (V\oplus V\oplus W)\otimes\oh_\yy\stackrel{\beta}{\longrightarrow} V\otimes\oh_\yy(1)
\end{equation}
given by
\begin{equation}
\label{alpha}
\alpha = \left( \begin{array}{c}
A' + \id\otimes x \\ B' + \id\otimes y \\ J
\end{array} \right) \ \ \ \
\beta = \left( \begin{array}{ccc}
-B - \id\otimes y ~~ & ~~ A + \id\otimes x ~~ & ~~ I
\end{array} \right).
\end{equation}
Given a point $P\in\yy$, we will denote by $\alpha_P$ and $\beta_P$ the fiber maps.

Note that $\beta\alpha=0$ iff $X$ satisfies the ADHM equation, which is a straightforward calculation left to the reader. Therefore, for such an $X$, we are able to do the following definition.

\begin{definition}
For any $X\in\calv_\yy$, we call $E^\bullet_{X}$ the \emph{ADHM complex over $\yy$ associated to $X$} and we refer to $E:=\ker\beta/{\rm im}\,\alpha$ as the \emph{cohomology sheaf of $E^\bullet_{X}$}.
\end{definition}

It is also important to check when the ADHM complex $E^\bullet_X$ happens to be a \emph{monad}, that is, when $\alpha$ is injective and $\beta$ surjective. We will see below that global weak stability is precisely the property required. Before that, let us just introduce some notation. Let $X$ be an ADHM datum; we denote by
$$
D_X:=\{ P\in\yy\ |\ \alpha_P\ \text{is not injective}\},
$$
which is the \emph{degeneration locus} of the map $\alpha$, and by
$$
d_X:=\codim_{\yy}(D_X)
$$
its codimension respect to $\yy$. We call $X\in\calv_{\yy}$ a \emph{nondegenerated} datum if $d_X\geq 2$, and say it is \emph{degenerated} otherwise. The set of nondegenerated points of the ADHM variety plays a central role in this work and is denoted by $\calv_{\yy}^{\bar{\rm n}}$.

\begin{remark}
\label{remdeg}
\emph{Note that for $n\geq 2$, one always has $\calv_{\pn}^{\bar{\rm n}}=\calv_{\pn}$, that is, there are no degenerated data over these projective spaces. More generally, if $\yy$ has dimension at least $2$ and ${\rm Pic}(\yy)=\mathbb{Z}$ then $\calv_{\yy}^{\bar{\rm n}}=\calv_{\yy}$ as well. On the other hand, for instance, Example \ref{exepxp} do provide the existence of degenerated points in the ADHM variety: if $X=((C\otimes z_0,C\otimes z_1,I),(0,0,0))\in\calv_{\rm S}$ where ${\rm S}$ is the scroll $z_0y-z_1x=0$ in $\p3$, then the degeneration locus of $X$ is the line $x=y=0$, and hence $d_X=1$.}
\end{remark}

\begin{proposition}
\label{l1}
The following hold:
\begin{enumerate}
\item[(i)] $\alpha$ is injective;
\item[(ii)] $\alpha_P$ is injective for every $P\in\yy$ iff $X$ is globally weak costable;
\item[(iii)] $\beta$ is surjective iff $X$ is globally weak stable.
\end{enumerate}
\end{proposition}

\begin{proof}
It is easy to see that $\alpha_P$ is injective for all $P\in\ell$. This means that the localized map $\alpha_P$ may fail to be injective
only at a subvariety of $\yy$ which does not intersect $\ell$, and hence $d_X\geq 1$. In particular, $\alpha$ is injective as a sheaf map and (i) follows.

To check (ii) and (iii), write $X=((A,B,I),(A',B',J))$. For (ii), if $X_P$ is not weak costable for some $P\in\yy$, there is a one dimensional subspace $S\subset V$ such that $A'_P(S),B'_P(S)\subset S\subset\ker J_P$. Fixing coordinates for $P$, we may assume that $(A'_P,B'_P,J_P)$ lies in $\aaa'$. It follows that there exists a nonzero $v\in V$, and $a,b\in\cpx$ such that
\begin{equation}
\label{equabi}
\begin{array}{lcr}
A'_P(v)=av\ \ \ & B'_P(v)= bv &\ \ \ J_P(v)=0
\end{array}
\end{equation}
and since the last coordinates of $P$ does not affect costability, we may suppose $P=(p_0:\dots:p_d:a:b)$ and hence the fiber map $\alpha_P$ is not injective. Conversely, if for such a
$P\in\yy$ the map $\alpha_P$ is not injective, then (\ref{equabi}) holds for a nonzero $v\in V$ and thus $S=\langle v\rangle$ makes $X_P$ non weak costable.

Finally, $\beta$ is surjective iff $\beta_P$ is surjective for all $P\in\yy$, which holds iff $\beta^t_P$ is injective for all $P\in\yy$, which, by the prior item, holds iff $(B_P^t,A_P^t,I_P^t)$ is weak costable for all $P\in\yy$. Now it is easily seen that $(B_P^t,A_P^t,I_P^t)$ is weak costable iff $(A_P,B_P,I_P)$ is weak stable and (iii) holds.
\end{proof}

For the next result, we recall some definitions. A projective scheme is called \emph{arithmetically Cohen-Macaulay} (ACM) if its homogeneous coordinate ring is a Cohen-Macaulay ring. The \emph{charge} of a coherent sheaf $E$ on a projective scheme is the integer $h^1(E(-1))$.

\begin{proposition}
\label{l4}
Let $X\in\calv_\yy(r,c)$ and $E$ be the cohomology sheaf of the ADHM complex $E^\bullet_{X}$. Then the following hold:
\begin{itemize}
\item[(i)] $E$ is a coherent sheaf which restricts trivially to $\ell$;
\item[(ii)] $E$ is a torsion free sheaf if and only if $X$ is nondegenerated;
\item[(iii)] $E^\bullet_{X}$ is a monad if and only if $X$ is globally weak stable;
\item[(iv)] $E$ is a locally free sheaf if and only if $X$ is globally weak costable;
\item[(v)] If $X$ is nondegenerated globally weak stable and $\yy$ is either $\p2$ or ACM of dimension at least $3$, then $E$ is a torsion free sheaf of rank $r$ and charge $c$.
\end{itemize}
\end{proposition}

\begin{proof}
The degeneration locus $D_X$  agrees with the singularity locus of $E$, that is, the points $P\in\yy$ for which the stalk $E_P$ is not a free $\oh_P$-module. Thus, it follows from \cite[Sec II.1.1]{OS} that $E$ is coherent iff $d_X\geq 1$, which always hold as seen in the proof of the prior proposition, and $E$ is torsion free iff $d_X\geq 2$. So the first statement of (i) and (ii) are proved. In order to check the other statement of (i), note that the restriction of (\ref{monad.pn}) to $\ell$ yields
$$
0 \to V\otimes\oh_{\ell}(-1)\stackrel{\alpha_{\ell}}{
\longrightarrow}
(V\oplus V\oplus W)\otimes\oh_{\ell}
\stackrel{\beta_{\ell}}{\longrightarrow}
V\otimes\oh_{\ell}(1) ~~$$
where
$$
\alpha_{\ell}=\left(\begin{array}{c}x\\y\\0\end{array}\right)\ \ \ \  \beta_{\ell}=\left(\begin{array}{ccc}-y~&~x~&~0\end{array}\right).
$$
Its cohomology, which coincides with the restriction of $E$ to $\ell$,
is just $W\otimes\oh_{\ell}$, so (i) is proved. Itens (iii) and (iv) hold due to Proposition \ref{l1} while (v) holds owing to \cite[Prp. 3.2]{JVM}.
\end{proof}


\subsection{The inverse construction}
\label{invconst}

In this subsection we will do the inverse construction, that is, build an ADHM datum out of a cohomology sheaf of a monad in $\yy$. First, we specify the relevant class of sheaves.

\begin{definition}
\label{defins}
A coherent sheaf $E$ on $\yy$ is called an \emph{instanton sheaf} if there exists a complex of the form
$$
\oh_{\yy}(-1)^{\oplus c} \stackrel{\alpha}{\longrightarrow}
\oh_{\yy}^{\oplus a} \stackrel{\beta}{\longrightarrow} \oh_{\yy}(1)^{\oplus c}
$$
where $\alpha$ is injective, $\beta$ is surjective and $E=\ker\beta/{\rm im}\,\alpha$. If $E$ restricts trivially to $\ell$, then $E$ is also called \emph{of trivial splitting type}.
\end{definition}

The above definition does agree with the known one in the case of projective spaces (see \cite[p. 69]{J-i}) as can be easily derived from \cite[Prp. 2 and Thm. 3]{J-i}.

\begin{proposition}
\label{prpmyy}
If $E$ is a torsion free instanton sheaf on $\yy$ of trivial splitting type with respect to $\ell$, then $E$ is the cohomology sheaf of a monad of the form
$$
E^\bullet_{X} ~:~ V\otimes\oh_{\yy}(-1) \stackrel{\alpha}{\longrightarrow}
(V\oplus V \oplus W)\otimes\oh_{\yy} \stackrel{\beta}{\longrightarrow} V\otimes\oh_{\yy}(1)
$$
where $W\simeq\ho^0(E|_{\ell})$ and $X\in\calv^{\bar{\rm n}{\rm gws}}_\yy(W,V)$, i.e., $X$ is a nondegenerated globally weak stable datum in the ADHM variety. Moreover, if $\yy$ is either $\p2$ or an ACM variety of dimension at least $3$, then $V\simeq \ho^1(E(-1))$.
\end{proposition}

\begin{proof} Since $E$ is instanton, it is the cohomology sheaf of a monad of the form
$$
V\otimes\oh_{\yy}(-1) \stackrel{\alpha}{\longrightarrow}
U\otimes\oh_{\yy} \stackrel{\beta}{\longrightarrow} V\otimes\oh_{\yy}(1).
$$
Restricting it to $\ell$, we get
\begin{equation}
\label{moo}
V\otimes\oh_{\ell}(-1) \stackrel{\alpha_\ell}{\longrightarrow}
U\otimes\oh_{\ell} \stackrel{\beta_\ell}{\longrightarrow} V\otimes\oh_{\ell}(1).
\end{equation}
The maps $\alpha_\ell$ and $\beta_\ell$ can then be expressed in the following manner:
\begin{align*}
\alpha_\ell &=\alpha_1 x + \alpha_2 y \\
\beta_\ell  &=\beta_1 x + \beta_2 y
\end{align*}
where $\alpha_k\in{\rm Hom}(V,U)$ and $\beta_k\in{\rm Hom}(U,V)$ for each
$k=1,2$. The condition $\beta_\ell\alpha_\ell=0$ then implies that
\begin{equation}
\label{equab1}
\beta_1\alpha_1=\beta_2\alpha_2=0
\end{equation}
\begin{equation}
\label{equab2}
\beta_1\alpha_2+\beta_2\alpha_1=0.
\end{equation}
From (\ref{moo}) we get the exact sequence
$$
0\longrightarrow V\otimes{\mathcal O}_{\ell}(-1)
\stackrel{\alpha_\ell}{\longrightarrow}
\ker\beta_{\ell} \longrightarrow E|_{\ell}\longrightarrow 0.
$$
Now $\ho^p(\oh_{\ell}(-1))=0$, for $p=0,1$, and $E|_{\ell}\simeq \oh_{\ell}^{\oplus {\rm rk}\,E}$, hence $\ho^1(\ker\beta_{\ell})=0$ and
\begin{equation}
\label{equkel}
\ho^0(\ker\beta_{\ell})\simeq \ho^0(E|_{\ell})\simeq E_P
\end{equation}
for some $P\in\ell$. Notice that the choice of a basis for $\ho^0(\ker\beta_{\ell})$ corresponds to the choice of a trivialization for $E|_{\ell}$.

Similarly, from (\ref{moo}) we also get the exact sequence
$$ 0 \longrightarrow \ker\beta_\ell \longrightarrow U\otimes\oh_{\ell}
\stackrel{\beta_\ell}{\longrightarrow} V\otimes\oh_{\ell}(1) \longrightarrow 0 $$
from which we obtain the short exact sequence of linear spaces
\begin{equation}
\label{m2}
0 \longrightarrow \ho^0(\ker\beta_\ell) \longrightarrow U \stackrel{\beta_{\ell}}{\longrightarrow}
V\otimes \ho^0(\oh_{\ell}(1)) \longrightarrow 0
\end{equation}
because $\ho^0(\oh_{\ell})\simeq\cpx$ and $\ho^1(\ker\beta_\ell)=0$. Since $\ho^0(\oh_{\ell}(1))\simeq \cpx^2$ we rewrite (\ref{m2}) as
\begin{equation}
\label{wtw}
0 \longrightarrow \ho^0(\ker\beta_\ell) \longrightarrow U \stackrel{\beta_1\oplus\beta_2}{\longrightarrow} V\oplus V \longrightarrow 0.
\end{equation}

Now $E$ is locally free on a neighborhood of $\ell$, so we can apply the same argument to the dual monad
$$
0 \longrightarrow V^*\otimes\oh_{\ell}(-1) \stackrel{\beta_\ell^{\rm t}}{\longrightarrow} U^*\otimes\oh_{\ell}
\stackrel{\alpha_\ell^{\rm t}}\longrightarrow V^*\otimes\oh_{\ell}(1) \longrightarrow 0
$$
of which we derive the exact sequence
\begin{equation}
\label{equuvd}
0 \longrightarrow \ho^0(\ker \alpha_\ell^{\rm t})
\longrightarrow U^* \stackrel{\alpha_1^{\rm t}\oplus\alpha_2^{\rm t}}{\longrightarrow} V^*\oplus V^* \longrightarrow 0
\end{equation}
with $\ho^0(\ker\alpha_{\ell}^{\rm t})^*\simeq \ho^0(E|_{\ell})\simeq \ho^0(\ker\beta_{\ell})$. Dualizing (\ref{equuvd}) we get
\begin{equation}
\label{equuvn}
0 \longrightarrow V\oplus V \stackrel{\alpha_1\oplus\alpha_2}{\longrightarrow}  U  \longrightarrow\ho^0(\ker\beta_\ell)\longrightarrow 0,
\end{equation}
hence (\ref{wtw}) splits and $U\simeq V\oplus V \oplus W$ where $W=\ho^0(\ker\beta_{\ell})\simeq \ho^0(E|_{\ell})$ as desired. Then $E$ is the cohomology sheaf of the monad
$$
V\otimes\oh_{\yy}(-1) \stackrel{\alpha}{\longrightarrow}
(V\oplus V\oplus W)\otimes\oh_{\yy}\stackrel{\beta}{\longrightarrow}
V\otimes\oh_{\yy}(1)
$$
described above. Now choose $P=(0:\ldots :0:1:0)$ to write (\ref{equkel}) as
$$
\ker\beta_1\cap\ker\beta_2=\ho^0(\ker\beta_\ell) \simeq E_P=\ker\beta_1/{\rm im}\alpha_1.
$$
Thus ${\rm im}\alpha_1\cap\ker\beta_2=0$, so that $\beta_1\alpha_2=-\beta_2\alpha_1:V\to V$  are isomorphisms. Therefore, we can choose basis to write
$$
\begin{array}{lcr}
\alpha_{1} = \left( \begin{array}{c}
\id \\ 0 \\ 0 \end{array} \right)

&

\alpha_{2} = \left( \begin{array}{c}
0 \\ \id \\ 0 \end{array} \right)

&

\begin{array}{c}

\beta_1 = \left( \begin{array}{ccc}
0 ~~ & ~~ \id ~~ & ~~ 0
\end{array} \right) \\

\\

\beta_2 = \left( \begin{array}{ccc}
-\id ~~ & ~~ 0 ~~ & ~~ 0
\end{array} \right)
\end{array}
\end{array}
$$
and the description of $\alpha$ and $\beta$ as coming from a datum $X$ easily follows. Moreover, $X$ needs to satisfy the ADHM equation since  $\beta\alpha=0$; by Proposition \ref{l4}.(ii)-(iii), it has to be globally weak stable since $\beta$ is surjective and nondegenerated since $E$ is torsion free.

Finally, the last claim follows easily from the short exact sequences
$$
0 \longrightarrow K \longrightarrow (V\oplus V \oplus W)\otimes\oh_{\yy} \stackrel{\beta}{\longrightarrow} V\otimes\oh_{\yy}(1)\longrightarrow 0
$$
$$
0\longrightarrow V\otimes\oh_{\yy}(-1) \stackrel{\alpha}{\longrightarrow} K \longrightarrow E \longrightarrow 0
$$
and we are done.
\end{proof}

The following concept is crucial for establishing the correspondence between data and torsion free instanton sheaves.

\begin{definition}
A \emph{framed instanton sheaf} on $\yy$ is a pair $(E,\phi)$ where $E$ is an instanton sheaf of trivial splitting type on $\yy$ and $\phi:E|_{\ell}\to\oh^{\oplus r}_{\ell}$ is an isomorphism.
\end{definition}

Two framed instanton sheaves $(E,\phi)$ and $(E',\phi')$ are said to be isomorphic if there is a sheaf isomorphism $\Phi:E\longrightarrow E$ and a constant $\lambda\in\mathbb{C}^{\ast}$ such that the following diagram is commutative:
\begin{displaymath}
\xymatrix@1{{E|_{\ell}}\ar[r]^{\Phi|_{\ell}}\ar[d]_{\phi}& {E'|_{\ell}}\ar[d]_{\phi'} \\
\mathcal{O}^{\oplus r}_{\ell}\ar[r]^{\lambda}&\mathcal{O}^{\oplus r}_{\ell}
}
\end{displaymath}

On the other hand, two ADHM data in $\mathbf{B}_\yy(W,V)$ are said to be \emph{equivalent} if their orbits modulo the action of $G=GL(V)$ coincide.

\begin{theorem}
\label{prpoto}
Let $\yy$ be either $\mathbb{P}^2$ or an ACM projective variety of dimension at least $3$.
Then there is a 1-1 correspondence between the following sets:
\begin{itemize}
\item equivalence classes of nondegenerated globally weak stable data in the ADHM variety $\calv_\yy(r,c)$;
\item isomorphism classes of rank $r$ framed torsion free instanton sheaves of charge $c$ on $\yy$.
\end{itemize}
\end{theorem}

\begin{proof}
For $g\in GL(V)$ and $h\in GL(W)$, consider the following data in $\calv_\yy(W,V)$:
$$
X=((A,B,I),(A',B',J))
$$
$$
X'=((gAg^{-1},gBg^{-1},gIh^{-1}),(gA'g^{-1},gB'g^{-1},hJg^{-1}))
$$
Then the map between ADHM complexes on $\yy$, given by
\begin{displaymath}
\xymatrix@1{{E_X^{\bullet}\,:\, 0}\ar[r] & {V\otimes\oh_{\yy}}\ar[r]^{\alpha\ \ \ \ \ \ \ \ \  }\ar[d]_{g\otimes 1} & {(V\oplus V\oplus W)\otimes\oh_{\yy}}\ar[d]_{(g\oplus g\oplus h)\otimes\id}\ar[r]^{\ \ \ \ \ \ \beta} & {V\otimes\oh_{\yy}}\ar[r]\ar[d]_{g\otimes 1} & {0}\\
{E_{X'}^{\bullet}\, :\, 0}\ar[r] & {V\otimes\oh_{\yy}}\ar[r]^{\alpha'\ \ \ \ \ \ \ \ \  }& {(V\oplus V\oplus W)\otimes\oh_{\yy}}\ar[r]^{\ \ \ \ \ \ \beta'} & {V\otimes\oh_{\yy}}\ar[r] & {0}&
}
\end{displaymath}
is an isomorphism. Conversely, any isomorphism between ADHM complexes are of the above form. But we have already seen in Proposition \ref{prpmyy} that any torsion free instanton sheaf of rank $r$ and charge $c$ is the cohomology of an ADHM complex associated to a nondegenerated globally weak stable datum in $\calv_{\yy}$ and conversely by Proposition \ref{l4}.(v). Moreover, isomorphisms between cohomology sheaves lift in an unique way to isomorphisms between monads since the conditions of \cite[Lem. 4.1.3]{OSS} are satisfied for linear monads on both $\mathbb{P}^2$ and ACM varieties. So, in order to get the result, we just observe that an $h\in GL(W)$ is nothing but a choice of framing for the cohomology sheaf of an ADHM complex.
\end{proof}

In particular, Theorem \ref{prpoto} provides a (set-theoretical) bijection between the set of isomorphism classes of rank $r$ framed torsion free instantons sheaves of charge $c$ on $\yy$ and points of $\calm^{\bar{\rm n}{\rm gws}}_{\yy}(r,c)$. So we have the following.

\begin{corollary}
The moduli space of isomorphism classes of rank $r$ framed torsion free instanton sheaves of charge $c$ on $\yy$ is a quasi projective variety.
\end{corollary}

We now turn our attention to the case of  $\yy=\pn$.


\section{Instantons on Projective Spaces}
\label{secmod}

We devote this section to the particular case of projective spaces, that is, the case where $\yy=\pn$ with $n=d+2\ge2$. When so, as seen in Example \ref{exappn}, any datum $X\in\calv_{\pn}$ is of the form $X=((A,B,I),(A,B,J))$, so it is more convenient to write
$$
X=(A,B,I,J)\in \mathbf{B}'\otimes\ho^0(\oh_{\pd}(1))
$$
where
$$
\bb'=\bb'(W,V)=\bb'(r,c):={\rm End}(V)^{\oplus 2}\oplus{\rm Hom}(W,V)\oplus {\rm Hom}(V,W).
$$
The ADHM equation reduces to
\begin{equation}
\label{equadp}
[A,B]+IJ=0.
\end{equation}
Moreover, weak global (resp. local) stability (corresp. costability) agrees with global (resp. local) stability (corresp. costability) in such a case. Indeed, write $X=(A,B,I,J)\in\calv_{\pn}$, fix $P\in\pn$ and assume $X_P$ is not costable. Hence there is a nonzero subspace $S\subset V$ such that $A_P(S),B_P(S)\subset S\subset\ker J_P$. Fixing coordinates for $P$, we may assume that $X_P=(A_P,B_P,I_P,J_P)$ lies in $\bb'$. Now, since $X\in\calv_{\pn}$, we have that $A_P|_S$ and $B_P|_S$ commute because $J_P|_S=0$. In particular, $A_P|_S$ and $B_P|_S$ are simultaneously diagonalizable. In particular, $A_P$ and $B_P$ have a common eigenvector in $S$ so $X_P$ is not weak costable. Hence global (resp. local) costability agrees with global (resp. local) costability. A similar statement holds for stability because $X_P$ is (resp. weak) stable iff $X^t_P:=(B_P^t,A_P^t,J_P^t,I_P^t)$ is (resp. weak) costable for all $P\in\pn$. Besides, there are no degenerated data in $\calv_{\pn}$.

For the sake of simplicity, we adopt the notation
$$
\calv_d:=\calv_{\mathbb{P}^{d+2}}=\calv_{\mathbb{P}^{n}}
$$
and, as seen above, it is enough to consider the sequence of proper inclusions
$$
\calv_d^{\rm gr} \subset \calv_d^{\rm gs} \subset \calv_d^{\rm ls} \subset \calv_d^{\rm st}.
$$


As well as for the ADHM variety in projective spaces, we adopt the same notation
$$
\calm_d:=\calm_{\mathbb{P}^{d+2}}=\calm_{\mathbb{P}^{n}}
$$
for the ADHM quotient variety of stable orbits in which we naturally consider the sequence
$$
\calm_d^{\rm gr} \subset \calm_d^{\rm gs} \subset \calm_d^{\rm ls} \subset \calm_d^{\rm st}.
$$
The simplest case, when $d=0$, is well known and can be found, for instance, in \cite{N2}:
${\mathcal M}_{0}^{\rm st}(r,c)$, which agrees with ${\mathcal M}_{0}^{\rm gs}(r,c)$, is an irreducible, nonsingular quasi-projective variety of dimension $2rc$, and it admits a complete hyperk\"ahler metric. In \cite{FJ2}, it is also proved that ${\mathcal M}^{\rm gs}_{1}(r,1)$ is irreducible and nonsingular, which we generalize to any $d$ in the following proposition.

\begin{proposition}
\label{dimension}
$\mathcal{M}_{d}^{\rm gs}(r,1)$ is a nonsingular quasi-projective variety of dimension $2(d+1)r-\frac{1}{2}d(d-1)$ if $r>d$ and empty otherwise.
\end{proposition}

\begin{proof}
Write a datum $X=(A,B,I,J)$ where the components are
$$
A = A_{0}\otimes z_0 + \cdots + A_{d}\otimes z_d\ \ \ \ \ \ \ \
B = B_{0}\otimes z_0 + \cdots + B_{d}\otimes z_d
$$
$$
I = I_0\otimes z_0 + \cdots + I_d\otimes z_d\ \ \ \ \ \ \ \
J = J_0\otimes z_0 + \cdots + J_d\otimes z_d\
$$
In this case, since $c=1$, the $A_k,B_k$ in are in $\cpx$, while the $I_k$ (resp. $J_k$) can be regarded as row (resp. column) matrix vectors in $\cpx^r$. Global stability reduces to the condition that the $I_k$ are linearly independent, so $\mathcal{M}_{d}^{\rm gs}(r,1)$ is empty if $r\leq d$. The group $G=\cpx^*$ acts trivially on the $A_{k},B_{k}$ and by multiplication by $t$ (resp. $t^{-1}$) on the $I_k$ (resp. $J_k$). The ADHM equations reduce to
$$
\begin{array}{cl}
I_kJ_k=0  & \ \  \ k=0,\dots, d \\

I_kJ_m + I_mJ_k = 0  &\ \ \  k<m=0,\dots,d
\end{array}
$$

It is then easily seeing that $\mathcal{M}_d^{\rm gs}(r,1)=\cpx^{2(d+1)}\times{\mathcal B}(d,r)$, where ${\mathcal B}(d,r)$ is the (open) set of solutions of the ADHM equations in the $(2(d+1)r-1)$-dimensional weighted projective space
$$ \begin{array}{ccc}
\mathbb{P} (\underbrace{1,\dots\dots,1} ,&\underbrace{-1,\dots\dots,-1}) \\
\ \ (d+1)r & (d+1)r \end{array} $$
such that the $I_k$ are linearly independent. This shows that ${\mathcal M}_d^{\rm gs}(r,1)$ is quasi-projective. The ADHM equations also yield the ${{d+2}\choose{2}}\times 2(d+1)r$ Jacobian matrix
$$
\left(
\begin{array}{cccccccccc}
\ddots &       &       &       &        &\ddots &       &           &         &          \\
       & J_k^t &       &       &        &       & I_k   &           &         &          \\
       &       &\ddots &       &        &       &       &  \ddots   &         &          \\
       &       &       & J_m^t &        &       &       &           & I_m     &          \\
       &       &       &       & \ddots &       &       &           &         & \ddots   \\
       &       &\vdots &       &        &       &       &   \vdots  &         &          \\
       & J_m^t &       & J_k^t &        &       & I_m   &           & I_k     &          \\
       &       &\vdots &       &        &       &       &  \vdots   &         &          \\
\end{array}
\right)
$$
which is of maximal rank if and only if the $I_k$ are linearly independent.
This shows that ${\mathcal M}_d^{\rm gs}(r,1)$ is nonsingular of the desired dimension.
\end{proof}

\subsection{The Moduli of Framed Instanton Bundles on $\mathbb{P}^{n}$}

In this section we study the moduli functor problem of framed instanton bundles on $\mathbb{P}^{n}$ by means of their monadic description.

Theorem \ref{prpoto} is an important step forward to get a moduli space for isomorphism classes of framed instanton bundles on projective spaces, which is the main concern of this subsection. In order to do so, we first introduce few notation: for every two sheaves, $\mathcal{F}$ on $\mathbb{P}^{n}$ and $\mathcal{G}$ on a scheme $S$, we put $\mathcal{F}\boxtimes \mathcal{G}:=p^{\ast}\mathcal{F}\otimes q^{\ast}\mathcal{G},$ where $p:\mathbb{P}^{n}\times S\longrightarrow\mathbb{P}^{n}$ is the projection on the first factor and $q$ is the projection $\mathbb{P}^{n}\times S\longrightarrow S$ on the second one. We also denote by $k(s)$ the residue field of a closed point $s\in S$.

Now, we start by introducing a moduli functor. Let
$$
\mathfrak{M}^{\mathbb{P}^n}_{r, c}: \mathfrak{S}ch \longrightarrow \mathfrak{S}et
$$
be the functor from the category of noetherian schemes of finite type to the category of sets, which is defined as follows:
to every scheme $S\in Obj(\mathfrak{S}ch)$ we associate the set
$$
\mathfrak{M}^{\mathbb{P}^{n}}_{r, c}(S)=\{\,\text{equivalence classes of pairs}\ (\mathcal{F}, \phi)\,\}
$$
where we have:
\begin{itemize}
\item[(i)] $\mathcal{F}$ is a coherent sheaf on $\mathbb{P}^{n}\times S$ which is flat on $S$;
\item[(ii)] $\phi:\mathcal{F}|_{\ell\times S}\to\mathcal{G}$ is an morphism (called \emph{framing}) to a certain sheaf $\mathcal{G}$ on $\ell\times S$ which is flat on $S$ and such that $\mathcal{G}\otimes k(s)\cong\mathcal{O}^{\oplus r}_{\ell}$ for every $s\in S$;
\item[(iii)] $(\mathcal{F}_s, \phi(s))$ is a framed instanton bundle on $\mathbb{P}^{n}$ of rank $r$ and charge $c$ for every $s\in S$, where $\mathcal{F}_s:=\mathcal{F}\otimes k(s)$ and $\phi(s):=\phi\otimes k(s)$;
\item[(iv)] $(\mathcal{F},\phi)\sim (\mathcal{F}',\phi')$ if there is a line bundle $L$ on $S$ such that $\mathcal{F}'\cong\mathcal{F}\otimes q^{\ast}L,$ and such that the following diagram commutes:
$$\xymatrix{\mathcal{F}\otimes q^{\ast}L|_{\ell\times S}\ar[r]\ar[d]_{\phi} & \mathcal{F}'|_{\ell\times S}\ar[d]^{\phi'}\\ \mathcal{G}\ar[r]& \mathcal{G}'}$$
\end{itemize}

The pull-back is defined as follows. Given a morphism $f:S'\to S$ and $[(\mathcal{F},\phi)]\in\mathfrak{M}^{\mathbb{P}^{n}}_{r,c}(S)$, we set
$$
f^{*}([(\mathcal{F},\phi)]):=[(({\rm id}_{\pn}\times f)^{*}\mathcal{F},\phi^{*})]
$$
where we naturally define
$$
\phi^*:({\rm id}_{\ell}\times f)^{*}(\mathcal{F}|_{\ell})\longrightarrow\mathcal{G}^*.
$$
This turns $\mathfrak{M}^{\mathbb{P}^n}_{r, c}$ into a contravariant family functor.

\

Now, the key point will be the use of the relative version of the Beilinson's Theorem. We remark that most of this subsection is a generalization of Le Potier's techniques \cite{Le}, which he used in order to describe the moduli space of stable bundles of rank 2 on $\mathbb{P}^{2}$. Le Potier's proof can also be found in \cite[Chp.II, \S4]{OSS}, and its generalization to the case framed torsion-free sheaves on multi-blow-ups of the projective plane can be found in \cite{henni}.

Let us consider a scheme $S$ which is noetherian and of finite type. Consider the following diagram:
\begin{displaymath}
\xymatrix@1{\mathbb{P}^{n}\times\mathbb{P}^{n}\times S \ar[r]^{\quad pr_{13}}\ar[d]_{pr_{23}} & \mathbb{P}^{n}\times S \ar[d]^{q} \\
\mathbb{P}^{n}\times S \ar[r]_{q}& S
}
\end{displaymath}
and the relative Euler sequence:
$$
0\longrightarrow\mathcal{O}_{\mathbb{P}^{n}\times S}(-1)\longrightarrow\mathcal{O}^{\oplus(n+1)}_{\mathbb{P}^{n}\times
S}\longrightarrow Q\boxtimes\mathcal{O}_{S}\longrightarrow 0
$$
where $Q=T\mathbb{P}^{n}(-1)$ is the twisted tangent bundle. One has the following:

\

\noindent{\bf Relative Beilinson's Theorem.} \emph{For every coherent sheaf $\mathcal{F}$ on $\mathbb{P}^{n}\times S$ there is a spectral sequence $E^{i,j}_{r}$ with $E_{1}$-term $$E_{1}^{i,j}=\mathcal{O}_{\mathbb{P}^{n}}(i)\boxtimes\mathcal{R}^{j}q_{\ast}(\mathcal{F}\otimes\Omega_{\mathbb{P}^{n}\times S/S}^{-i}(-i))$$ which converges to $$E_{\infty}^{i,j}=\left\{\begin{array}{ll} \mathcal{F}& i+j=0 \\ 0 & {\rm otherwise.} \end{array}\right.$$}

\

The above result is the tool we need to have the following.

\begin{theorem}
\label{thmfin}
The quasiprojective scheme $\calm_{d}^{\rm gr}(r,c)$ of globally regular solutions to the $d$-dimensional ADHM equation modulo the action of $GL(c)$ is a fine moduli space for the isomorphism classes of rank $r$ framed instantons bundles of charge $c$ on $\mathbb{P}^{d+2}$.
\end{theorem}

\begin{proof}
We first construct the natural transformation
$$
\mathfrak{M}^{\mathbb{P}^{n}}_{r,c}(\bullet)\longrightarrow\Hom(\bullet,\calmdgr(r,c)),
$$
where $n:=d+2$, between the moduli functor of framed instanton bundles on $\mathbb{P}^{n}$ and the Yoneda functor associated to the solutions to the  $d$-dimensional ADHM equation modulo the action of $G$.

So let $S$ be a scheme and take $[(\mathcal{F},\phi)]\in\mathfrak{M}^{\mathbb{P}^{n}}_{r,c}(S)$. We claim that $\mathcal{F}$ is the cohomology of a monad
\begin{equation}
\label{equsmo}
{\rm M}^{\bullet}:\ \ \ \mathcal{O}_{\mathbb{P}^{n}}(-1)\boxtimes\mathcal{R}^{1}q_{\ast}(\mathcal{F}\otimes\Omega_{\mathbb{P}^{n}\times S/S}^{2}(1))\longrightarrow\mathcal{O}_{\mathbb{P}^{n}}\boxtimes\mathcal{R}^{1}q_{\ast}(\mathcal{F}\otimes\Omega_{\mathbb{P}^{n}\times S/S}^{1})
\end{equation}
$$
\longrightarrow\mathcal{O}_{\mathbb{P}^{n}}(1)\boxtimes \mathcal{R}^{1}q_{\ast}(\mathcal{F}\otimes p^{\ast}\mathcal{O}_{\mathbb{P}^{n}}(-1)).
$$
In fact, by the Relative Beilinson's Theorem, we just use the fact that $\mathcal{F}$ is $S$-flat, so that
$$\mathcal{R}^{j}q_{\ast}(\mathcal{F}\otimes\Omega_{\mathbb{P}^{n}\times S/S}^{-i}(-i))\otimes k(s)\simeq\ho^{j}(\mathbb{P}^{n},\mathcal{F}_{s}\otimes\Omega_{\mathbb{P}^{n}}^{-i}(-i)).
$$
Then (\ref{equsmo}) follows by using the vanishing properties of instanton bundles.

Therefore, on every point $s\in S$, one has a monad
$$
{\rm M}^{\bullet}_s\,:\,\ho^{1}(\mathcal{F}_{s} \otimes\Omega_{\mathbb{P}^{n}}^{2}(1))\otimes\mathcal{O}_{\mathbb{P}^{n}}(-1)
\to
\ho^{1}(\mathcal{F}_{s}\otimes\Omega_{\mathbb{P}^{n}}^{1})\otimes\mathcal{O}_{\mathbb{P}^{n}}
\to
\ho^{1}(\mathcal{F}_{s}(-1))\otimes\mathcal{O}_{\mathbb{P}^{n}}(1).
$$

Now, consider an open covering $\{S_{j}\}_{j\in J}$ of $S.$ Then on every open affine $S_{j}$ the restriction ${\rm M}^{\bullet}|_{S_{j}}$ is isomorphic to a monad of the form
$$
{\rm M}^{\bullet}_j\,:\ \ \ \mathcal{O}_{\mathbb{P}^{n}}(-1)\boxtimes (V\otimes\mathcal{O}_{S_{j}})\stackrel{\alpha_{j}}{\longrightarrow}  \mathcal{O}_{\mathbb{P}^{n}}\boxtimes (\tW\otimes\mathcal{O}_{S_{j}})\stackrel{\beta_{j}}{\longrightarrow}\mathcal{O}_{\mathbb{P}^{n}}(1)\boxtimes (V'\otimes\mathcal{O}_{S_{j}})
$$
where
\begin{align*}
\alpha_{j}: & \ S_{j}\longrightarrow \Hom(\ho^{0}(\mathcal{O}_{\mathbb{P}^{n}}(1))^*, \Hom(V,\tW)) \\
\beta_{j}: & \ S_{j}\longrightarrow \Hom(\ho^{0}(\mathcal{O}_{\mathbb{P}^{n}}(1))^*, \Hom(\tW,V')).
\end{align*}
From the monad condition $\beta_{j}\circ\alpha_{j}=0$ we have a map $f_{j}=(\alpha_{j},\beta_{j}):S_{j}\to\calvdgr$ and by construction these maps satisfy $$f_{i}(s)\sim_{G} f_{j}(s)$$ for any point $s$ in the intersection $S_{i}\cap S_{j}.$ The maps $f_{j}$ glue to form a global morphism
$$
f=f_{[(\mathcal{F},\phi)]}:S\longrightarrow \calmdgr.
$$
This defines the desired natural transformation:
\begin{gather}
\label{natural-transform}
\begin{matrix}
\Phi(\bullet) : & \mathfrak{M}^{\mathbb{P}^{n}}_{r,c}(\bullet) & \longrightarrow & \Hom(\bullet,\calmdgr(r,c))  \\
                             & \chi & \longmapsto     & f_{\chi}:\bullet\longrightarrow \calmdgr .
\end{matrix}
\end{gather}

Taking a closed point $s\in S$, and using the resulting monad on $\mathbb{P}^{n}$ it is easy to see that $\Phi: \mathfrak{M}^{\mathbb{P}^{n}}_{r,c}(\spec k(s))\longrightarrow \Hom(\spec k(s),\calmdgr(r,c))$ is a bijection owing to Theorem \ref{prpoto}.

Finally, take $\mathcal{N}$ to be another parameterizing scheme such that there is a natural transformation
$$
\Psi: \mathfrak{M}^{\mathbb{P}^{n}}_{r,c}(\bullet)\longrightarrow \Hom(\bullet,\mathcal{N}),
$$
and consider the monad
\begin{equation}
\label{equuni}
\mathbb{M}^{\bullet}: \mathcal{O}_{\mathbb{P}^{n}}(-1)\boxtimes (V\otimes\mathcal{O}_{\mathcal{V}^{\rm gr}_d})\to  \mathcal{O}_{\mathbb{P}^{n}}\boxtimes (\tW\otimes\mathcal{O}_{\mathcal{V}^{\rm gr}_d})\to\mathcal{O}_{\mathbb{P}^{n}}(1)\boxtimes (V\otimes\mathcal{O}_{\mathcal{V}^{\rm gr}_d})
\end{equation}
of which the cohomology we call $\mathfrak{F}$. We first claim that the map
$$
\psi:=\Psi(\calvdgr)_{[(\mathfrak{F},\phi)]}: \calvdgr\longrightarrow \mathcal{N}
$$
is constant along the fibers of the natural projection
$$
\pi=\Phi(\calvdgr)_{[(\mathfrak{F},\phi)]}: \calvdgr\longrightarrow \calmdgr
$$
for any framing $\phi$. Actually, the assertion is not particular for the chosen scheme and family. Rather, it easily comes from the fact that $\mathfrak{M}^{\mathbb{P}^{n}}_{r,c}$  is a contravariant family functor and  $\Psi$ and $\Phi$ are natural transformations, that is, the square diagrams obtained fom $\Psi$ and $\Phi$  by pull-backing families and composing morphisms are all commutative.

The projection $\pi:\calvdgr\longrightarrow \calmdgr$ locally has sections, so one can construct local mappings $\bar{\varphi}:\calmdgr\longrightarrow \mathcal{N}$, but since $\psi$ is constant along the fibers of $\pi$, then the map $\bar{\varphi}$ can be lifted to a global map $\varphi$ such that the following diagram commutes:

\begin{equation}
\label{equcom}
\xymatrix@C-0.5pc@R-0.5pc{\calvdgr\ar[r]^{\psi}\ar[d]_{\pi}& \mathcal{N}\\
\calmdgr\ar[ru]_{\varphi}&
}
\end{equation}

Again, from the very naturality of $\Phi$ and $\Psi$ we have that the natural morphism $\varphi:\calmdgr\to\mathcal{N}$ as in (\ref{equcom}) is enough to get a unique natural transformation
$$
\Omega : {\rm Hom}(\bullet,\calmdgr)\longrightarrow{\rm Hom}(\bullet,\mathcal{N})
$$
such that $\Psi=\Omega\circ\Phi$. Hence $\mathcal{M}^{\mathbb{P}^{n}}_{r,c}$ is a coarse moduli space.

To finish the proof, we shall now descend the universal monadic description on $\mathbb{P}^{n}\times \mathcal{V}^{\rm gr}_{d}$ to a well behaved monadic description on $\mathbb{P}^{n}\times\mathcal{M}^{\rm gr}_{d}$. This can be realized due to the fact that the space $\mathbb{P}^{n}\times \mathcal{V}^{\rm gr}_{d}$ is a $G$-space since there is a natural action
\begin{gather}
\begin{matrix}
G\times\mathbb{P}^{n}\times \mathcal{V}^{\rm gr}_{d} & \longrightarrow & \mathbb{P}^{n}\times \mathcal{V}^{\rm gr}_{d}  \\
(g,(x,X)) & \longmapsto     & (x,g\cdot X)
\end{matrix}
\end{gather}
This induces a $G$-action on the universal monad $\mathbb{M}$, in \eqref{equuni}, which descends to an action on its cohomology $\mathfrak{F}$, but since the action is free and the isotropy subgroup is trivial at all points owing to Proposition \ref{Sh1}, we have a well defined family $\mathfrak{F}/G\longrightarrow\mathbb{P}^{n}\times \mathcal{V}^{\rm gr}_{d}/G$. We put $\mathfrak{U}:=\mathfrak{F}/G$ which is a canonical family $$\mathfrak{U}\longrightarrow\mathbb{P}^{n}\times\mathcal{M}^{\rm gr}_{d}$$ parameterized by $\mathcal{M}^{\rm gr}_{d}$.

Finally, we claim that for any noetherian  scheme $S$ of finite type, the mapping
$$\begin{array}{cccc}\Hom(S,\mathcal{M}^{gr}_{d}) & \longrightarrow & \mathfrak{M}^{\mathbb{P}^{n}}_{(r,c)}(S)\\
\phi & \longmapsto & \phi^{\ast}[\mathfrak{U}]=[({\rm id}_{\mathbb{P}^{n}}\times\phi)^{\ast}\mathfrak{U}]\end{array}$$ is bijective.

In fact, for injectivity, if there are homomorphisms $\phi_{1}, \phi_{2}: S\longrightarrow\mathcal{M}^{\rm gr}_{d}$ such that $({\rm id}_{\mathbb{P}^{n}}\times\phi_{1})^{\ast}\mathfrak{U}\cong({\rm id}_{\mathbb{P}^{n}}\times\phi_{2})^{\ast}\mathfrak{U}$
then for every point $s\in S$, one has the equality $\mathfrak{U}(\phi_{1}(s))=\mathfrak{U}(\phi_{2}(s))$. Since the bundle $\mathfrak{U}(\phi_{i}(s))$ is the one given by the globally regular ADHM data associated to the point $\phi_{i}(s)\in \mathcal{M}^{\rm gr}_{d}$, then $\phi_{1}(s)=\phi_{2}(s)$ for every point $s\in S$, thus $\phi_{1}=\phi_{2}.$

For surjectivity, given a family $\mathcal{F}$ parameterized by $S,$ one has the morphism $\phi=\Phi(\mathcal{F})$ given by the natural transformation \eqref{natural-transform}. Then $\mathcal{F}$ is the pull-back of the family $\mathfrak{U}$ parameterized by $\mathcal{M}^{\rm gr}_{d}$. Hence $\mathcal{M}^{\rm gr}_{d}$ is a fine moduli space.
\end{proof}


\section{Perverse instanton sheaves}\label{perverse}

We will conclude this paper by providing a geometrical interpretation for arbitrary solutions of the ADHM equation as perverse coherent sheaves on $\yy$. Indeed, as remarked in Section \ref{secfra}, arbitrary solutions of the ADHM equation give rise to the complex of sheaves (\ref{monad.pn}) which, thought as an object of the derived category $D^{\rm b}(\yy)$, is a perverse coherent sheaf with very particular properties.


\subsection{t-structures and perverse sheaves}

Let $\calt$ be a triangulated category. We recall from \cite{GM} that a \emph{t-structure} on $\calt$ consists of two full subcategories, denoted $D^{\le0}$ and $D^{\ge0}$ satisfying the following conditions:
\begin{itemize}
\item[(i)] $D^{\le0} \subset D^{\le0}[-1]$ and $D^{\ge0}[-1] \subset D^{\ge0}$;
\item[(ii)] if $\mathcal{X}\in D^{\le0}$ and $\mathcal{Z}\in D^{\ge0}$, then $\hom_{\mathcal T}(\mathcal{X},\mathcal{Z})=0$;
\item[(iii)] for every $\mathcal{Y}\in\calt$, there exists an exact triangle $\mathcal{X}\to \mathcal{Y}\to \mathcal{Z}\to \mathcal{X}[1]$ with $\mathcal{X}\in D^{\le0}$ and $\mathcal{Z}\in D^{\ge0}$.
\end{itemize}

The full subcategory of $\calt$ consisting of those objects in $D^{\le0}\cap D^{\ge0}$ is called the \emph{core} of the t-structure $(D^{\le0},D^{\ge0})$; one can show that it is always an abelian category, see \cite{GM}.

If $\calt$ is the (bounded) derived category of an abelian category $\cala$, then one can define the so-called standard t-structure:
$$
{}^{\rm std}D^{\le0} = \{ C^\bullet \in D(\cala) ~|~ \calh^k(C^\bullet)=0 ~ \forall k>0 \}\
$$
$$
{}^{\rm std}D^{\ge0} = \{ C^\bullet \in D(\cala) ~|~ \calh^k(C^\bullet)=0 ~ \forall k<0 \}.
$$
One easily checks that $({}^{\rm std}D^{\le0},{}^{\rm std}D^{\ge0})$ is indeed a t-structure, and that its core is equivalent to $\cala$.

The main goal of this section is to outline two methods of construction of non-standard t-structures on derived categories of coherent sheaves on projective varieties. From now on, let $\xx$ denote a non-singular irreducible projective variety over an algebraically closed field.


\subsection{Kashiwara's t-structures}\label{kashiwara}

Let ${\rm Mod}(\ox)$ denote the abelian category of sheaves of $\ox$-modules, and set $D(\ox)$ to be its derived category; let also $D(\xx)$ to be the derived category of ${\rm Coh}(\xx)$. As usual, we set $D_{\rm qc}(\ox)$ ($D_{\rm coh}(\ox)$) to be the full triangulated subcategory of $D(\ox)$ consisting of complexes with quasi-coherent (coherent) cohomology. Recall that $D^{\rm b}(\xx)$ is naturally equivalent to $D_{\rm coh}^{\rm b}(\ox)$.

We will also use the costumary notation $D_{\rm qc}^{\le n}(\ox)$ to mean complexes $C^\bullet$ in $D_{\rm qc}(\ox)$ such that $\calh^k(C^\bullet)=0$ for all $k>n$; similarly, $D_{\rm qc}^{\ge n}(\ox)$ means complexes $C^\bullet$ in $D_{\rm qc}(\ox)$ such that $\calh^k(C^\bullet)=0$ for all $k<n$. The curly cohomology $\calh$ is used to denote the cohomology of a complex in $D(\ox),$ which is an $\mathcal{O}_{\xx}-$module, while the straight $\ho$ denotes the cohomology with respect to the global sections functor $\Gamma$, in the category of vector spaces.

A \emph{family of supports} on $\xx$ is a set $\Phi$ of closed subsets of $\xx$ satisfying the following conditions: (i) if $Z\in\Phi$ and $Z'$ is a closed subset of $Z$, then $Z'\in\Phi$; (ii) if $Z,Z'\in\Phi$, then $Z\cup Z'\in\Phi$; (iii) $\emptyset\in\Phi$. For any such family of supports, consider the functor $\Gamma_{\Phi}:{\rm Mod}(\ox) \to {\rm Mod}(\ox)$ defined as follows:
$$ \Gamma_{\Phi}(F) := \lim_{Z\in\Phi} \Gamma_Z(F). $$
Then one has, for each open subset $U\subset \xx$:
\begin{equation}\label{gamma}
\Gamma_{\Phi}(F)(U) = \{ \sigma\in F(U) ~|~ \overline{{\rm supp}\sigma} \in \Phi \}.
\end{equation}

A \emph{support datum} on $\xx$ is a decreasing sequence $\mathbf{\Phi}:=\{\Phi^n\}_{n\in\Z}$ of families of supports satisfying the following conditions: (i) for $n\ll0$, $\Phi^n$ is the set of all closed subsets of $X$; (ii) for $n\gg0$, $\Phi^n=\{\emptyset\}$.

Finally, the \emph{support} (or \emph{perversity}) \emph{function} associated to the support datum $\mathbf{\Phi}$ (see \cite[Lem. 5.5]{Ka}) is:
\begin{gather*}
\begin{matrix}
p_{\mathbf{\Phi}}: & \xx_{\rm top} & \longrightarrow & \Z  \\
                             & x & \longmapsto     & {\rm max}\{ n\in\Z ~|~ \overline{\{x\}}\in\Phi^n \}
\end{matrix}
\end{gather*}
where $\xx_{\rm top}$ denotes the topological space underlying the natural scheme structure on the variety $\xx$.

Given a support datum on $\xx$, Kashiwara introduces the following subcatgories of $D^{\rm b}_{\rm qc}(\ox)$:
$$
{}^{\mathbf{\Phi}}D_{\rm qc}^{\le n}(\ox) := \left\{ C^\bullet \in D^{\rm b}_{\rm qc}(\ox) ~|~ {\rm supp}(\calh^k(C^\bullet)) \in \Phi^{k-n} ~\forall k \right\}\ \ \ \ \,
$$
$$
{}^{\mathbf{\Phi}}D_{\rm qc}^{\ge n}(\ox) := \left\{ C^\bullet \in D^{\rm b}_{\rm qc}(\ox) ~|~ R\Gamma_{\Phi^k}(C^\bullet) \in D^{\ge k+n}(\ox) ~\forall k \right\}.
$$
Now consider as in \cite[p. 857]{Ka}:
$$
{}^{\mathbf{\Phi}}D_{\rm coh}^{\le 0}(\ox) := {}^{\mathbf{\Phi}}D_{\rm qc}^{\le 0}(\ox) \cap D^{\rm b}(\xx)\,
$$
$$
{}^{\mathbf{\Phi}}D_{\rm coh}^{\ge 0}(\ox) := {}^{\mathbf{\Phi}}D_{\rm qc}^{\ge 0}(\ox) \cap D^{\rm b}(\xx).
$$
It is shown in \cite[Thm. 5.9]{Ka} that if the support function $p_\Phi$ satisfies the following condition
\begin{equation}\label{pervcond}
p_{\mathbf{\Phi}}(y) - p_{\mathbf{\Phi}}(x) \le \codim(\overline{\{y\}}) - \codim(\overline{\{x\}})~~ \forall y\in \overline{\{x\}},
\end{equation}
then $( {}^{\mathbf{\Phi}}D_{\rm coh}^{\le 0}(\ox),{}^{\mathbf{\Phi}}D_{\rm coh}^{\ge 0}(\ox) )$ defines a $t$-structure on $D^{\rm b}(\xx)$.

\begin{example}
\label{ex1}
\emph{For the scheme $\yy$ with the line $\ell$, consider the following support datum $\mathbf{\Phi}=\{\Phi^k\}_{k\in\Z }$ with
\begin{align*}
\Phi^k &:= \{ \rm all~closed~subsets~of~\yy \} ~~ {\rm for} ~ k\le0 \\
\Phi^1 &:= \{ \rm all~closed~subsets~of~\yy ~which~do~not~intersect~\ell \} \\
\Phi^k &:= \{ \emptyset \}  ~~ {\rm for} ~ k\ge2.
\end{align*}
The corresponding perversity function $p_\Phi : \yy_{\rm top} \to \Z$ is given by: $p_{\mathbf{\Phi}}(x)=0$ if and only if $\overline{\{x\}}\cap\ell\ne\emptyset$ and $p_{\mathbf{\Phi}}(x)=1$ otherwise. One easily checks that such function does satisfy the condition (\ref{pervcond}). We will denote by $\calc_\yy$ the core of this t-structure.}
\end{example}

The objects of $\calc_\yy$ can be characterized as follows.

\begin{proposition}\label{calc}
$C^\bullet\in \calc_\yy$ if and only if the following hold:
\begin{itemize}
\item[(i)] $\calh^k(C^\bullet)=0$ for $k\ne0,1$;
\item[(ii)] $\calh^0(C^\bullet)$ has all nonzero sections with support intersecting $\ell$;
\item[(iii)] $\calh^1(C^\bullet)$ is supported away from $\ell$.
\end{itemize}
\end{proposition}
\begin{proof}
We first assume that $C^\bullet\in \calc_\yy$. On the one hand, $C^\bullet \in {}^{\mathbf{\Phi}}D_{\rm coh}^{\le 0}(\oh_{\yy})$. Then ${\rm supp}(\calh^k(C^\bullet))$ is empty for $k\ge2$, i.e., $\calh^k(C^\bullet)=0$ for $k\ge2$, and also ${\rm supp}(\calh^1(C^\bullet))$ does not intersect $\ell$.

On the other hand, $C^\bullet \in {}^{\mathbf{\Phi}}D_{\rm coh}^{\ge 0}(\oh_{\yy})$. Then, first, $R\Gamma_{\Phi^0}(C^\bullet) \in D^{\ge 0}(\oh_{\yy})$. But $\Gamma_{\Phi^0}$ is just the identity functor on ${\rm Mod}(\oh_{\yy})$, thus $R\Gamma_{\Phi^0}(C^\bullet)=C^\bullet\in D^{\ge 0}(\oh_{\yy})$, i.e, $\calh^k(C^\bullet)=0$ for $k\leq -1$. Besides, $R\Gamma_{\Phi^1}(C^\bullet)\in D^{\ge 1}(\oh_{\yy})$ and hence $\calh^0(R\Gamma_{\Phi^1}(C^\bullet))$ vanishes. But, by \cite[Lem. 3.3.(iii)]{Ka}, $\calh^0(R\Gamma_{\Phi^1}(C^\bullet))=\Gamma_{\Phi^1}(\calh^0(C^\bullet))$ since $C^\bullet \in D^{\ge 0}(\oh_{\yy})$. Therefore $\Gamma_{\Phi^1}(\calh^0(C^\bullet))=0$ which is equivalent to saying (ii).

Conversely, let us first check that $C^\bullet\in{}^{\mathbf{\Phi}}D_{\rm coh}^{\le 0}(\oh_{\yy})$. This is quite clear, since ${\rm supp}(\calh^1(C^\bullet)) \in \Phi^{1}$ by (iii) and ${\rm supp}(\calh^k(C^\bullet))=\emptyset \in \Phi^{k}$ for $k\ge2$ by (i).

To check that $C^\bullet\in{}^{\mathbf{\Phi}}D_{\rm coh}^{\ge 0}(\oh_{\yy})$, since $C^\bullet \in D^{\ge 0}(\oh_{\yy})$ we use  \cite[Lem. 3.3.(iii)]{Ka} again. It gets $R\Gamma_{\Phi^k}(C^\bullet) \in D^{\ge 0}(\oh_{\yy})$ and $\calh^0(R\Gamma_{\Phi^k}(C^\bullet))=\Gamma_{\Phi^k}(\calh^0(C^\bullet))$ for every $k$. So, first, $R\Gamma_{\Phi^k}(C^\bullet) \in D^{\ge 0}(\oh_{\yy})\subset D^{\ge k}(\oh_{\yy})$ for every $k\le 0$. Besides, we have $R\Gamma_{\Phi^1}(C^\bullet) \in D^{\ge 0}(\oh_{\yy})$ and  $\calh^0(R\Gamma_{\Phi^1}(C^\bullet))=\Gamma_{\Phi^1}(\calh^0(C^\bullet))$ which vanishes by (ii); hence $R\Gamma_{\Phi^1}(C^\bullet) \in D^{\ge 1}(\oh_{\yy})$. Since $R\Gamma_{\Phi^k}(C^\bullet)=0$ for every $k\ge2$ ($\Phi^k=\{\emptyset\}$ in this range), we also have that $R\Gamma_{\Phi^k}(C^\bullet) \in D^{\ge k}(\oh_{\yy})$ for $k\ge 2$.
\end{proof}


\subsection{Tilting on torsion pairs}\label{tilting}

Let $\cala$ be an abelian category, and let $(\calt,\calf)$ be a pair of full subcategories of $\cala$. One says that $(\calt,\calf)$ is a torsion pair in $\cala$ if the following conditions are satisfied:
\begin{itemize}
\item[(i)] $\rm{Hom}_{\cala}(T,F)=0$ whenever $T\in\calt$ and $F\in\calf$;
\item[(ii)] For every $A\in\cala$, there is a short exact sequence $0\to T\to A\to F\to 0$ with $T\in\calt$ and $F\in\calf$.
\end{itemize}

Now let $(\calt,\calf)$ be a torsion pair in $\cala$. Consider the full subcategories of $D^{\rm b}(\cala)$:
$$
D^{\le0} := \{ C^\bullet\in D^{\rm b}(\cala) ~|~ \calh^p(C^\bullet)=0 ~{\rm for}~p>0~{\rm and}~\calh^0(C^\bullet)\in\calt \}\ \ \ \ \
$$
$$
D^{\ge0} := \{ C^\bullet\in D^{\rm b}(\cala) ~|~ \calh^p(C^\bullet)=0 ~{\rm for}~p<-1~{\rm and}~\calh^{-1}(C^\bullet)\in\calf \}.
$$
According to \cite[Prp. 2.1 and Cor. 2.2]{HRS}, we have that $(D^{\le0},D^{\ge0})$ is a t-structure on $D^{\rm b}(\cala)$, and the full subcategories $(\calf[1],\calt)$ is a torsion pair in its core. In this situation, one says that the t-structure (and its core) is obtained from $\cala$ through \emph{tilting} on the torsion pair $(\calt,\calf)$.

\begin{example}
\label{ex2}
\emph{For the scheme $\yy$, take $\cala=\rm{Coh}(\yy)$; given a coherent sheaf $E$ on $\yy$, let $T_k(E)$ be the maximal subsheaf of $E$ whose support has dimension at most $k$, see \cite[p. 3]{HL}. Consider the following full subcategories of $\rm{Coh}(\yy)$:
$$
\calt: = \{ E\in{\rm Coh}(\yy) ~|~ T_{n-2}(E)=E \}
$$
$$
\calf: = \{ E\in{\rm Coh}(\yy) ~|~ T_{n-2}(E)=0 \}.
$$
One easily checks that they form a torsion pair in $\rm{Coh}(\yy)$. Let $\calb$ be the core of the t-structure obtained from $\rm{Coh}(\yy)$ through tilting on $(\calt,\calf)$. We set
$$
\calc'_{\yy}:=\calb[1]
$$}
\end{example}

The objects of $\calc'_{\yy}$ can easily be characterized as follows.

\begin{proposition}
\label{calc'}
$C^\bullet\in \calc'_\yy$ if and only if the following hold:
\begin{itemize}
\item[(i)] $\calh^p(C^\bullet)=0$ for $p\ne0,1$;
\item[(ii)] $\calh^0(C^\bullet)$ has no subsheaves supported in codimension at least $2$;
\item[(iii)] $\calh^1(C^\bullet)$ is supported in codimension at least $2$.
\end{itemize}
\end{proposition}

We remark that the two t-structures obtained in Examples \ref{ex1} and \ref{ex2} are distinct, and so are their cores. For instance, $\mathcal{O}_{\ell}[1]$ is an object of $\calc'_\yy$, but not of $\calc_\yy$; on the other hand, the sheaf $\mathcal{O}_{\ell}$ is an object in $\calc_\yy$, but not of $\calc'_\yy$.


\subsection{Perverse instanton sheaves}

Broadly speaking, a perverse coherent sheaf on an algebraic variety $\mathbb{X}$ is an object within the core of some t-structure on $D^{\rm b}(\mathbb{X})$. Therefore, motivated by the above examples, we introduce the following definition.

\begin{definition}\label{pervdef}
A \emph{perverse (coherent) sheaf} on $\yy$ is a complex $C^\bullet\in D^b(\yy)$ satisfying the following conditions:
\begin{itemize}
\item[(i)] $\calh^p(C^\bullet)=0$ for $p\ne0,1$;
\item[(ii)] $\calh^0(C^\bullet)$ is a torsion free sheaf;
\item[(iii)] $\calh^1(C^\bullet)$ is a torsion sheaf supported away from a line $\ell$.
\end{itemize}
The \emph{rank} $r$ of $C^\bullet$ is defined to be the rank of $\calh^0(C^\bullet)$.
\end{definition}

Let $\calp_{\yy}$ denote the category of perverse sheaves on $\yy$, as a full subcategory of $D^b(\yy)$. It is easy to see that $\calp_{\yy}$ is additive, closed under direct summands and closed under extensions.
Moreover, $\calp_{\yy}$ is contained both in $\calc_{\yy}$ and in $\calc'_{\yy}$, and it contains the category of torsion-free sheaves on $\yy$ as a subcategory.

Let $\calf_{\yy}$ be the category of torsion free sheaves on $\yy$ as a subcategory of  $\calp_{\yy}$, i.e., $F^\bullet\in\calf_{\yy}$ if $\calh^1(F^\bullet)=0$, and let $\calz_{\yy}$ be the category of rank zero perverse sheaves, i.e., $Z^\bullet\in\calz_{\yy}$ if $\calh^0(Z^\bullet)=0$. Note that $(\calf_{\yy},\calz_{\yy})$ is a torsion pair in $\calp_{\yy}$; in particular, for every $C^\bullet\in\calp_{\yy}$, there is a short exact sequence
$$ 0 \to F^\bullet \to C^\bullet \to Z^\bullet \to 0 $$
with $F^\bullet\in\calf_{\yy}$ and $Z^\bullet\in\calz_{\yy}$.

One can then extend Definition \ref{defins} from coherent to perverse sheaves.

\begin{definition}
An object $C^{\bullet}$ in ${\rm Kom}(\yy)$ is said to be a \emph{perverse instanton sheaf} if it is quasi-isomorphic to a complex of the form
$$
\oh_{\yy}(-1)^{\oplus c} \longrightarrow
\oh_{\yy}^{\oplus a} \longrightarrow \oh_{\yy}(1)^{\oplus c}
$$
such that $Lj^{\ast}C^{\bullet}$ is a sheaf object, where $j:\ell\hookrightarrow\yy$
is the inclusion. If the sheaf object $Lj^{\ast}C^{\bullet}$ on $\ell$ is trivial, then $C^{\bullet}$ is called of trivial splitting type. A \emph{framed perverse instanton sheaf} is the pair $(C^{\bullet},\phi)$ consisting of a perverse instanton sheaf $C^{\bullet}$ of trivial splitting type and a framing $\phi: Lj^{\ast}C^{\bullet} \simto \oh_\ell^{\oplus r}$.
\end{definition}

We point out that perverse instanton sheaves may fail to be perverse sheaves. In fact, if $X$ is a datum in the ADHM variety then $E_X^{\bullet}$ is, by construction, a perverse instanton sheaf, but if $X$ is degenerated then $\mathcal{H}^0(E_X^{\bullet})$ is not torsion free, in particular $E_X^{\bullet}$ is not a perverse sheaf, and we have seen in Remark \ref{remdeg} examples of degenerated data. On the other hand, if for instance ${\rm Pic}(\yy)=\mathbb{Z}$ then one may adjust the proof of \cite[Prp. 2.7]{HL} to conclude that perverse instanton sheaves on $\yy$ are always perverse sheaves.

\begin{proposition}
\label{surjective}
Let $\yy$ be such that ${\rm Pic}(\yy)=\mathbb{Z}$, then a complex is a perverse instanton sheaf of trivial splitting type on $\yy$ if and only if it is quasi-isomorphic to an ADHM one.
\end{proposition}

\begin{proof}
Given a complex $C^{\bullet}\in D^{\rm b}(\yy)$, first we assure that $\mathcal{H}^0(C^{\bullet})$ is torsion free using the fact that ${\rm Pic}(\yy)=\mathbb{Z}$ and adjusting \cite[Prp. 2.7]{HL}. Then we apply verbatim the proof of Proposition \ref{prpmyy} without obliging the complex to be a monad.
\end{proof}


\subsection{Functorial point of view}

The correspondence indicated in the previous proposition can also be described in terms of a functor. First, we construct the {\em ADHM category over} $\yy$ which we denote $\mathfrak{A}(\yy)$. The objects of $\mathfrak{A}(\yy)$ are triples $(V,W,X)$ with $X\in \calv_{\yy}(W,V)$, and a morphism
$$
\rho : (V,W,X)\longrightarrow (V',W',X')
$$
consists of two linear maps $f:V\to V'$ and $g:W\to W'$ such that if we write
$$
X=(X_i)_{i=1}^6\in{\rm Hom}(V,V\otimes \ho_\yy)^{\oplus 4}\oplus {\rm Hom}(W,V\otimes \ho_\yy)\oplus{\rm Hom}(V,W\otimes \ho_\yy)
$$
$$
X'=(X_i')_{i=1}^6\in{\rm Hom}(V',V'\otimes \ho_\yy)^{\oplus 4}\oplus {\rm Hom}(W',V'\otimes \ho_\yy)\oplus{\rm Hom}(V',W'\otimes \ho_\yy)
$$
then the diagrams
$$
\xymatrix{ V \ar[r]^{X_i\ \ \ \ \ } \ar[d]^{f} &
\, V\otimes {\rm H}_{\yy} \ar[d]^{f\otimes 1} \\
V' \ar[r]^{X_i'\ \ \ \ \ } & \, V'\otimes {\rm H}_{\yy} }
$$
are commutative for $1\leq i\leq 4$, and the diagrams
$$
\xymatrix{ W \ar[r]^{X_5\ \ \ \ \ } \ar[d]^{g} &
\, V\otimes {\rm H}_{\yy} \ar[d]^{f\otimes 1} \\
W' \ar[r]^{X_5'\ \ \ \ \ } & \, V'\otimes {\rm H}_{\yy} }
\ \ \ \ \ \ \ \ \ \ \ \ \ \ \ \ \ \ \ \
\xymatrix{ V \ar[r]^{X_6\ \ \ \ \ } \ar[d]^{f} &
\, W\otimes {\rm H}_{\yy} \ar[d]^{g\otimes 1} \\
V' \ar[r]^{X_6'\ \ \ \ \ } & \, W'\otimes {\rm H}_{\yy} }
$$
are commutative as well.

Note that $\mathfrak{A}(\p2)$ is the category of representations of the ADHM quiver
$$
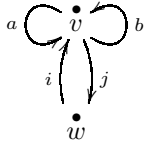
\xymatrix{
\stackrel{\bullet}{v} \ar@/^/[d]^{j} \ar@(ul,dl)[]_{a} \ar@(dr,ur)[]_{b} \\
\stackrel{\bullet}{w} \ar@/^/[u]^{i}}
$$
with the relation $ab-ba+ij=0$. It is also possible to use the notion of twisted representations of quivers in the sense \cite{GK,J-pn} to describe $\mathfrak{A}(\yy)$ for a general $\yy$. Note also that $\mathfrak{A}(\yy)$ is abelian. We denote by $\mathfrak{S}(\yy)$ the full subcategory of $\mathfrak{A}(\yy)$ whose objects are globally weak stable if regarded as ADHM data.

On the other hand, let $\calp{\mathcal I}(\yy)$ denote the full subcategory of $D^b(\yy)$ whose objects are perverse instanton sheaves of trivial splitting type; let also ${\mathcal I}(\yy)$  denote the full subcategory of ${\rm Coh}(\yy)$ consisting of instanton sheaves of trivial splitting type.
In the sequel we want to establish a relation between the categories $\mathfrak{A}(\yy)$ and $\calp{\mathcal I}(\yy)$. In order to do so, we start by a lemma.

\begin{lemma}
\label{lemh00}
Let $\yy$ be such that $\ho^1(\oh_{\yy}(-i))=0$ for $i=1,2$. Let also $\phi :E^{\bullet}\to F^{\bullet}$ be a morphism of ADHM complexes over $\yy$. Then the following hold:
\begin{itemize}
\item[(i)] if $\mathcal{H}^0(F^{\bullet})=\mathcal{H}^1(\phi)=0$ then $\phi=0$;
\item[(ii)] if $\mathcal{H}^1(E^{\bullet})=\mathcal{H}^0(\phi)=0$ then $\phi=0$;
\end{itemize}
\end{lemma}

\begin{proof}
First, write
$$
E^\bullet ~:~ \oh_\yy(-1)^{\oplus c} \stackrel{\alpha}{\longrightarrow}\oh_\yy^{\oplus a}\stackrel{\beta}{\longrightarrow} \oh_\yy(1)^{\oplus c}
$$
$$
F^\bullet ~:~ \oh_\yy(-1)^{\oplus c'} \stackrel{\alpha'}{\longrightarrow}\oh_\yy^{\oplus a'}\stackrel{\beta'}{\longrightarrow} \oh_\yy(1)^{\oplus c'}
$$
and break down $\phi$ into the following three morphisms between exact sequences
\begin{equation}
\label{equmo1}
\xymatrix{
0 \ar[r] &\oh_\yy(-1)^{\oplus c}\ar[r]^{\alpha} \ar[d]^{\phi_{-1}} & \ker\beta \ar[r] \ar[d]^{\phi_0|_{\ker\beta}}  & \mathcal{H}^0(E^{\bullet}) \ar[d]^{\mathcal{H}^0(\phi)} \ar[r] & 0  & ({\rm I}) \\
0 \ar[r] & \oh_\yy(-1)^{\oplus c'} \ar[r]^{\alpha'} & \ker\beta' \ar[r] &  \mathcal{H}^0(F^{\bullet}) \ar[r] & 0 & ({\rm II})}
\end{equation}
\begin{equation}
\label{equmo2}
\xymatrix{
0 \ar[r] &\ker\beta\ar[r]\ar[d]^{\phi_0|_{\ker\beta}} & \oh_\yy^{\oplus a} \ar[r]^{\beta} \ar[d]^{\phi_0}  & {\rm im}\,\beta \ar[d]^{\phi_1|_{{\rm im}\,\beta}} \ar[r] & 0 & \ \ \ \ \ \ \ \ \ \ ({\rm I}) \\
0 \ar[r] & \ker\beta'\ar[r]  & \oh_\yy^{\oplus a'} \ar[r]^{\beta'} & {\rm im}\,\beta'  \ar[r] & 0 & \ \ \ \ \ \ \ \ \ \ ({\rm II})}
\end{equation}
\begin{equation}
\label{equmo3}
\xymatrix{
0 \ar[r] &{\rm im}\,\beta\ar[r] \ar[d]^{\phi_1|_{{\rm im}\,\beta}} & \oh_\yy(1)^{\oplus c} \ar[r] \ar[d]^{\phi_1}  & \mathcal{H}^1(E^{\bullet}) \ar[d]^{\mathcal{H}^1(\phi)} \ar[r] & 0 & ({\rm I}) \\
0 \ar[r] & {\rm im}\,\beta' \ar[r] & \oh_\yy(1)^{\oplus c'} \ar[r] & \mathcal{H}^1(F^{\bullet})  \ar[r] & 0. & ({\rm II}) }
\end{equation}
Applying ${\rm Hom}(\oh_{\yy}(1)^{\oplus c},\bullet)$ to (\ref{equmo3}.II) one gets the exact sequence
\begin{equation}
\label{equsq1}
{\rm Hom}(\oh(1)^{\oplus c},{\rm im}\,\beta' )\longrightarrow{\rm Hom}(\oh(1)^{\oplus c}, \oh(1)^{\oplus c'}) \stackrel{\gamma}{\longrightarrow} {\rm Hom}(\oh(1)^{\oplus c},\mathcal{H}^1(F^{\bullet}))
\end{equation}
but we have
\begin{equation}
\label{equcg1}
{\rm Hom}(\oh_{\yy}(1)^{\oplus c},{\rm im}\,\beta' ) \simeq \ho^0({\rm im}\,\beta'(-1))^{\oplus c} \simeq \ho^1(\ker\beta'(-1))^{\oplus c}
\end{equation}
where the second isomorphism can be deduced from (\ref{equmo2}.II) since $\ho^1(\oh_{\yy}(-1))=0$ by hypothesis. On the other hand, ${\rm Hom}(\oh_{\yy}^{\oplus a},\bullet)$ applied to (\ref{equmo2}.II) yields
\begin{equation}
\label{equsq2}
{\rm Hom}(\oh_{\yy}^{\oplus a},\ker\beta' )\longrightarrow{\rm Hom}(\oh_{\yy}^{\oplus a}, \oh_{\yy}^{\oplus a'}) \stackrel{\delta}{\longrightarrow} {\rm Hom}(\oh_{\yy}^{\oplus c},{\rm im}\,\beta')
\end{equation}
but we also have
\begin{equation}
\label{equcg2}
{\rm Hom}(\oh_{\yy}^{\oplus a},\ker\beta' )\simeq\ho^0(\ker\beta')^{\oplus a}.
\end{equation}

Now assume $\mathcal{H}^0(F^{\bullet})=0$. Then $\ker\beta'\simeq\oh_{\yy}(-1)$ owing to (\ref{equmo1}.II). It implies first, by (\ref{equcg1}), that
${\rm Hom}(\oh_{\yy}(1)^{\oplus c},{\rm im}\,\beta' )=0$ since $\ho^1(\oh_{\yy}(-2))=0$ by hypothesis, so $\gamma$ is injective by (\ref{equsq1}). It  also implies, by (\ref{equcg2}), that ${\rm Hom}(\oh_{\yy}^{\oplus a},\ker\beta' )=0$, so $\delta$ is injective by (\ref{equsq2}). Assume also $\mathcal{H}^1(\phi)=0$. Then $\gamma(\phi_1)=0$ so $\phi_1=0$ because $\gamma$ is injective. Thus $\gamma(\phi_1)|_{{\rm im}\,\beta}=0$, so $\delta(\phi_0)=0$ and hence $\phi_0=0$ because $\delta$ is injective. In particular, $\phi_0|_{\ker \beta}=0$ which implies that $\phi_{-1}=0$ since $\alpha$ and $\alpha'$ are injective. Therefore $\phi=0$ and (i) is proved.

To prove (ii), apply ${\rm Hom}(\ker\beta,\bullet)$ to (\ref{equmo1}.II) and get the exact sequence
\begin{equation}
\label{equsq3}
{\rm Hom}(\ker\beta,\oh_{\yy}(-1)^{\oplus c'})\longrightarrow{\rm Hom}(\ker\beta, \ker\beta') \stackrel{\mu}{\longrightarrow} {\rm Hom}(\ker\beta,\mathcal{H}^0(F^{\bullet}))
\end{equation}
but we have
\begin{equation}
\label{equcg3}
{\rm Hom}(\ker\beta,\oh_{\yy}(-1)^{\oplus c'})\simeq \ho^0((\ker\beta)^{\vee}(-1))^{\oplus c'}.
\end{equation}
On the other hand, ${\rm Hom}(\bullet,\oh_{\yy}^{\oplus a'})$ applied to (\ref{equmo2}.I) yields
\begin{equation}
\label{equsq4}
{\rm Hom}({\rm im}\,\beta,\oh_{\yy}^{\oplus a'} )\to{\rm Hom}(\oh_{\yy}^{\oplus a},\oh_{\yy}^{\oplus a'}) \stackrel{\nu}{\longrightarrow} {\rm Hom}(\ker\beta,\oh_{\yy}^{\oplus a'})
\end{equation}
but we also have
\begin{equation}
\label{equcg4}
{\rm Hom}({\rm im}\,\beta,\oh_{\yy}^{\oplus a'} )\simeq \ho^0(({\rm im}\,\beta)^{\vee})^{\oplus a'}.
\end{equation}

Now assume $\mathcal{H}^1(E^{\bullet})=0$. Then ${\rm im}\,\beta=\oh_{\yy}(1)^{\oplus c}$. It implies first, by (\ref{equcg4}), that
${\rm Hom}({\rm im}\,\beta,\oh_{\yy}^{\oplus a'} )=0$ so $\nu$ is injective by (\ref{equsq4}). It also implies, dualizing (\ref{equmo2}.I) and tensorizing it by $\oh_{\yy}(-1)$, that $\ho^0((\ker\beta)^{\vee}(-1))=0$, so ${\rm Hom}(\ker\beta,\oh_{\yy}(-1)^{\oplus c'})$ vanishes by (\ref{equcg3}) and $\mu$ is injective by (\ref{equsq3}). Assume also $\mathcal{H}^0(\phi)=0$. Then $\mu(\phi_0|_{\ker\beta})=0$ so $\phi_0|_{\ker\beta}=0$ because $\mu$ is injective. So $\phi_{-1}=\nu(\phi_0)=0$. Hence $\phi_0=0$ because $\nu$ is injective. But if $\phi_0=0$ so is $\phi_1|_{{\rm im}\,\beta}$ which implies $\phi_1=0$ with our assumption. Therefore $\phi=0$ and (ii) is proved.
\end{proof}

Now we state the close relation between $\mathfrak{A}(\yy)$ and $\calp{\mathcal I}(\yy)$, as mentioned above.

\begin{theorem}
\label{thm111}
The following hold:
\begin{itemize}
\item[(i)] The assignment
\begin{gather*}
\begin{matrix}
\F : &{\rm Ob}(\mathfrak{A}(\yy)) & \longrightarrow & {\rm Ob}(\mathcal{PI}(\yy))\\
               &  (V,W,X)             & \longmapsto     & E^\bullet_X
\end{matrix}
\end{gather*}
defines a functor between the ADHM category and the category of perverse instantons sheaves of trivial splitting type.
\end{itemize}
If $\yy$ is such that ${\rm Pic}(\yy)=\mathbb{Z}$, then:
\begin{itemize}
\item[(ii)] $\F$ is essentially surjective.
\end{itemize}
If $\yy$ is either $\p2$ or an ACM variety of dimension at least $3$, it also holds:
\begin{itemize}
\item[(iii)] $\F$ is faithfull;
\end{itemize}

\end{theorem}

\begin{proof}
To prove (i), it is enough to define how the assignment $X\mapsto E^\bullet_X$ acts on morphisms. So let $\rho=\{f,g\}$ be a morphism between two triples $(V,W,X)$ and $(V',W',X')$; thus  $f:V\to V'$ and $g:W\to W'$. One then has the following morphism of complexes $\phi:E^\bullet_{X}\to E^\bullet_{X'}$ defined by
$$
\xymatrix{
V'\otimes\oh_{\yy}(-1) \ar[r]^{\alpha'\ \ \ \ \ \ \ } \ar[d]^{f\otimes\id} & (V'\oplus V'\oplus W')\otimes\oh_{\yy} \ar[r]^{\ \ \ \ \ \
\ \beta'} \ar[d]^{(f\oplus f\oplus g)\otimes\id} &
V'\otimes\oh_{\yy}(1) \ar[d]^{f\otimes\id} \\
V\otimes\oh_{\yy}(-1) \ar[r]^{\alpha\ \ \ \ \ \ \ } & (V\oplus V\oplus W)\otimes\oh_{\yy} \ar[r]^{\ \ \ \ \ \ \ \beta} & V\otimes\oh_{\yy}(1)
} . $$
So one defines $\F(\rho)$ to be the roof $E^\bullet_\mathbf{X}\stackrel{\id}{\leftarrow}E^\bullet_\mathbf{X}\stackrel{\phi}{\rightarrow}E^\bullet_{\mathbf{X}'}$ and (i) is proved.

The assertion (ii) is nothing but Proposition \ref{surjective} within the categorical framework just introduced. The proof of (iii) reduces to verifying the following statement: if a morphism $\phi: E^{\bullet}\to F^{\bullet}$ between ADHM complexes vanish in all cohomologies then $\phi$ is the zero morphism. But it suffices to prove this when either $\mathcal{H}^0(E^{\bullet})=\mathcal{H}^0(F^{\bullet})=0$ or $\mathcal{H}^1(E^{\bullet})=\mathcal{H}^1(F^{\bullet})=0$ since $\mathcal{PI}(\yy)$ splits into a torsion pair defined precisely by these two properties. Then one applies Lemma \ref{lemh00}.
\end{proof}

\begin{remark}\rm 
We do not know whether the functor $\F:\catA((\yy)\to\calp{\mathcal I}(\yy,\ell)$ is also full in the case of $\yy$ being either $\p2$ or an ACM variety of dimension at least $3$. It is not difficult to see that ${\rm Hom}_{{\rm K}(\yy)}( E^\bullet_{X_1}, E^\bullet_{X_2})$ is indeed isomorphic to ${\rm Hom}_{\catA((\yy)}(X_1,X_2)$; here, ${\rm K}(\yy)$ denotes the homotopy category. However, we do not know how to check whether the natural map 
$$ {\rm Hom}_{{\rm K}(\yy)}( E^\bullet_{X_1}, E^\bullet_{X_2})\to {\rm Hom}_{D^b(\yy)}( E^\bullet_{X_1}, E^\bullet_{X_2}) $$
is also surjective in general. This fact holds when both $X_1$ and $X_2$ are globally stable.
\end{remark}


\subsection{Perverse instanton sheaves on $\pn$}\label{pervpn}

In this subsection we characterize the case of projective spaces by means of  vanishing of some hypercohomologies.  In order to achieve our goal, first we prove the following version of Beilinson's Theorem.

\begin{theorem}
Let $C^{\bullet}$ be a perverse sheaf on $\mathbb{P}^{n}$, then there exists a spectral sequence $\mathbb{E}_{r}^{-p,q}$ with $E_{1}$-term of the form
\begin{equation}
\mathbb{E}_{1}^{-p,q}=\mathbb{H}^{q}(C^{\bullet}(p))\otimes\Omega_{\mathbb{P}^{n}}^{p}(p)
\end{equation}
for which the degree zero converges to
\begin{equation}
    \mathbb{E}_{\infty}^{i=0}=C^{\bullet}
\end{equation}
 where $i=q-p$ and $p\geq0.$
\end{theorem}

\begin{proof}
Let $C^{\bullet}$ be a perverse sheaf on $\mathbb{P}^{n}.$ The proof is a generalization of the Beilinson's Theorem \cite[Ch. II, \S 3]{OSS} to the case of sheaf complexes.

The Koszul resolution of the sheaf associated to the diagonal $\Delta\cong\mathbb{P}^{n}$ in $\mathbb{P}^{n}\times\mathbb{P}^{n}$ is given by the complex $\widetilde{K}^{\bullet}$ defined by
$$
\widetilde{K}^{-j}:=\quad \mathcal{O}_{\mathbb{P}^{n}}(-j)\boxtimes\Omega^{j}_{\mathbb{P}^{n}}(j)
$$
$$
d_{\widetilde{K}}^{j}:\mathcal{O}_{\mathbb{P}^{n}}(-j)\boxtimes\Omega^{j}_{\mathbb{P}^{n}}(j)\longrightarrow \mathcal{O}_{\mathbb{P}^{n}}(-j+1)\boxtimes \Omega^{j-1}_{\mathbb{P}^{n}}(j-1).
$$
Let $p_{1}$ and $p_{2}$ be the two natural projections from $\mathbb{P}^{n}\times\mathbb{P}^{n}$ to $\mathbb{P}^{n}$.
Twisting the resolution above by $p_{1}^{\ast}(C^{\bullet})$ one obtain the double complex $K^{\bullet\bullet}$ given by
$$
K^{i,-j}:=\quad C^{i}(-j)\boxtimes\Omega^{j}_{\mathbb{P}^{n}}(j)
$$
\begin{align*}
d_{K}^{j}:C^{i}(-j)\boxtimes\Omega^{j}_{\mathbb{P}^{n}}(j) &\longrightarrow C^{i}(-j+1)\boxtimes \Omega^{j-1}_{\mathbb{P}^{n}}(j-1)
\\
d_{C}^{i}:C^{i}(-j)\boxtimes\Omega^{j}_{\mathbb{P}^{n}}(j)&\longrightarrow C^{i+1}(-j)\boxtimes \Omega^{j}_{\mathbb{P}^{n}}(j).
\end{align*}
Let us denote by $(T^{\bullet}(K),D_{K})$ the total complex associated to the double complex $K^{\bullet\bullet}$. Then it is easy to see that the complexes $T^{\bullet}(K)$ and $C^{\bullet}|_{\Delta}$ are quasi-isomorphic, i.e., $[T^{\bullet}(K)]\cong[C^{\bullet}|_{\Delta}]$ in $D^{\rm b}(\mathbb{P}^{n}\times\mathbb{P}^{n}).$

Let $L^{\bullet\bullet\bullet}$ be a triple complex such that for each term $K^{i,-j}$ the complex
$$
0 \longrightarrow K^{i,-j} \longrightarrow L^{i,-j,1}\stackrel{D_1}{\longrightarrow}L^{i,-j,2}\stackrel{D_2}{\longrightarrow}\cdots L^{i,-j,k}\stackrel{D_k}{\longrightarrow}\cdots
$$
is an injective resolution, i.e., the complex $L^{\bullet\bullet\bullet}$ can be seen as a generalized Cartan-Eilemberg resolution in the category of bounded complexes $Kom^{\rm b}(\mathbb{P}^{n}\times\mathbb{P}^{n})$, where $L^{i,-j,0}=K^{i,-j}.$ The $k$-th hyperdirect images of the double complex $K^{\bullet\bullet}$, with respect to the projection $p_{2}$, can then be defined as in \cite{Grothendieck} by
$$
\mathcal{R}^{k}p_{2\ast}(K^{\bullet\bullet}):=\ho^{k}_{D}(p_{2\ast}(L^{\bullet\bullet\bullet})).
$$
Then one has two spectral sequences with $E_{2}-$terms
\begin{align*}
\mathbb{E}^{p,q}_{2}&=\ho^{p}_{d_{K}}(\mathcal{R}^{q}p_{2\ast}(K^{\bullet\bullet})) \\ '\mathbb{E}^{p,q}_{2}&=\mathcal{R}^{p}p_{2\ast}(\ho^{q}_{d_{K}}(K^{\bullet\bullet})).
\end{align*}
Let us compute the $'\mathbb{E}_{2}-$term. Since $K^{\bullet\bullet}$ is a resolution of $p^{\ast}_{1}(C^{\bullet})|_{\Delta}$, then one has
$$
\ho^{q}_{d_{K}}(K^{\bullet\bullet})=\left\{\begin{array}{ll} p^{\ast}_{1}(C^{\bullet})|_{\Delta}& \textnormal{for } q=0\\0 &\textnormal{otherwise}.\end{array}\right.
$$
It follows that
$$
'\mathbb{E}^{p,q}_{2}=\left\{\begin{array}{ll} \mathcal{R}^{p}p_{2\ast}(p^{\ast}_{1}(C^{\bullet})|_{\Delta})& \textnormal{for } q=0\\0 &\textnormal{otherwise} \end{array}\right.
$$
and hence
$$
'\mathbb{E}^{p,q}_{2}=\left\{\begin{array}{ll}C^{\bullet}& \textnormal{for } p=q=0\\0 &\textnormal{otherwise}. \end{array}\right.
$$
That is, the spectral sequence above degenerates at second step and converges to
$$
'\mathbb{E}^{p,q}_{\infty}=\left\{\begin{array}{ll}C^{\bullet}& \textnormal{for } p=q=0\\0 &\textnormal{otherwise}. \end{array}\right.
$$
The $E_{1}-$term of the other spectral sequence is given by
\begin{align}
\mathbb{E}^{-p,q}_{1}&=\mathcal{R}^{q}p_{2\ast}(p^{\ast}_{1}(K^{\bullet,-p})) \notag \\
&=\mathcal{R}^{q}p_{2\ast}(p^{\ast}_{1}(C^{\bullet}(-p)\boxtimes\Omega^{p}_{\mathbb{P}^{n}}(p))) \notag \\
&=\mathcal{R}^{q}p_{2\ast}(p^{\ast}_{1}(C^{\bullet}(-p))\otimes\Omega^{p}_{\mathbb{P}^{n}}(p)) \notag \\
&=\mathbb{H}^{q}(C^{\bullet}(-p))\otimes\Omega^{p}_{\mathbb{P}^{n}}(p) \notag
\end{align}
where in the last step above we just used the projection formula.
\end{proof}

Twisting the Koszul resolution by $p_{2}^{\ast}(C^{\bullet})$ and following the same reasoning, one also shows the following

\begin{theorem}
\label{beilinson-perverse}
Let $C^{\bullet}$ be a perverse sheaf on $\mathbb{P}^{n}$, then there exists a spectral sequence $\mathbb{E}_{r}^{-p,q}$ with $E_{1}$-term of the form
\begin{equation}
\mathbb{E}_{1}^{-p,q}=\mathbb{H}^{q}(C^{\bullet}\otimes\Omega_{\mathbb{P}^{n}}^{p}(p))\otimes\mathcal{O}_{\mathbb{P}^{n}}(-p)
\end{equation}
for which the degree zero converges to
\begin{equation}
    \mathbb{E}_{\infty}^{i=0}=C^{\bullet}
\end{equation}
 where $i=q-p$ and $p\geq0.$
\end{theorem}

We are finally ready to establish the cohomological characterization of perverse instanton sheaves on $\pn$ promised above.

\begin{theorem}
\label{thmfim}
Let $C^{\bullet}$ be a perverse sheaf on $\mathbb{P}^{n}.$ Consider the following statements:
\begin{itemize}
\item[(1)] $C^{\bullet}$ is a perverse instanton sheaf;
\item[(2)] $C^{\bullet}$ satisfies the following conditions:
\begin{itemize}
\item[(i)] for $n\geq2;$ $\mathbb{H}^{0}(C^{\bullet}(-1))=\mathbb{H}^{n}(C^{\bullet}(-n))=0$;
\item[(ii)] for $n\geq3;$ $\mathbb{H}^{1}(C^{\bullet}(-2))=\mathbb{H}^{n-1}(C^{\bullet}(1-n))=0$;
\item[(iii)] for $n\geq4;$ $\mathbb{H}^{p}(C^{\bullet}(k))=0$ $\forall k,$ for $2\leq p\leq n-2$.
\end{itemize}
\end{itemize}
Then $(2)$ implies $(1).$ Moreover if $C^{\bullet}$ is a perverse instanton sheaf of trivial splitting type, then the two statements are equivalent.
\end{theorem}

\begin{proof}
\noindent $(2)\Rightarrow(1)$. If $C^{\bullet}$ satisfies (2) we claim that it is quasi-isomorphic to the following complex
$$
\mathbb{H}^{1}(C^{\bullet}\otimes\Omega_{\mathbb{P}^{n}}^{2}(1))\otimes\mathcal{O}_{\mathbb{P}^{n}}(-1)\stackrel{\alpha}\longrightarrow \mathbb{H}^{1}(C^{\bullet}\otimes\Omega_{\mathbb{P}^{n}}^{1})\otimes\mathcal{O}_{\mathbb{P}^{n}}\stackrel{\beta}\longrightarrow \mathbb{H}^{1}(C^{\bullet}(-1))\otimes\mathcal{O}_{\mathbb{P}^{n}}(1).
$$

Indeed, the proof follows \cite[Thm. 3]{J-i}. Since hypercohomology is cohomological functor \cite{Grothendieck} it carries all the properties of the usual sheaf cohomology on short exact sequences (of complexes or of sheaves considered as complexes concentrated in degree zero) and their restriction to subschemes: given a hyperplane $L\subset\mathbb{P}^{n},$ consider the restriction sequence
$$
0\longrightarrow C^{\bullet}(-1)\longrightarrow C^{\bullet}\longrightarrow C^{\bullet}|_{L}\longrightarrow0
$$
Clearly, $\mathbb{H}^{0}(C^{\bullet}(-1))=0$ implies that $\mathbb{H}^{0}(C^{\bullet}(k))=0$ for $k\leq-1,$ while \linebreak $\mathbb{H}^{n}(C^{\bullet}(-n))=0$ forces $\mathbb{H}^{0}(C^{\bullet}(k))=0$ for $k\geq-n$. Since 
$$ \mathbb{H}^{0}(C^{\bullet}(-1))=\mathbb{H}^{1}(C^{\bullet}(-2))=0, $$ 
it follows that $\mathbb{H}^{0}(C^{\bullet}(-1)|_{L})=0,$ hence $\mathbb{H}^{0}(C^{\bullet}(k)|_{L})=0$ for $k\leq-1$. So we have the sequence
$$
0\longrightarrow\mathbb{H}^{1}(C^{\bullet}(k-1))\longrightarrow \mathbb{H}^{1}(C^{\bullet}(k))
$$
for $k\leq-2$, thus by induction $\mathbb{H}^{1}(C^{\bullet}(k))=0$ for $k\leq-2$.

Since $\mathbb{H}^{n}(C^{\bullet}(-n))=\mathbb{H}^{n-1}(C^{\bullet}(1-n))=0$, it follows that $\mathbb{H}^{n-1}(C^{\bullet}(1-n)|_{L})=0$, hence by further restriction $\mathbb{H}^{n-1}(C^{\bullet}(k)|_{L})=0$ for $k\geq1-n$. So we have the sequence
$$
\mathbb{H}^{n-1}(C^{\bullet}(k-1))\longrightarrow\mathbb{H}^{n-1}(C^{\bullet}(k))\longrightarrow0
$$
for $k\geq1-n$, thus by induction $\mathbb{H}^{n-1}(C^{\bullet}(k))=0$ for $k\geq1-n$.

We will show that $\mathbb{H}^{q}(C^{\bullet}(-1)\otimes\Omega^{p}_{\mathbb{P}^{n}}(p))=0$ for $q=1, p\geq3$. This follows from repeated use of the exact sequence
$$
\mathbb{H}^{q}(C^{\bullet}(k))^{\oplus m}\to \mathbb{H}^{q}(C^{\bullet}(k+1)\otimes\Omega^{p-1}_{\mathbb{P}^{n}}(p-1))\to \mathbb{H}^{q}(C^{\bullet}(k)\otimes\Omega^{p}_{\mathbb{P}^{n}}(p))\to \mathbb{H}^{q+1}(C^{\bullet}(k))^{\oplus m}
$$
associated with the Euler sequence for $p-$forms on $\mathbb{P}^{n}$ twisted by $C^{\bullet}(k),$ \cite[Ch.I,\S1]{OSS}:
$$
0\longrightarrow C^{\bullet}(-1)\otimes\Omega^{p}_{\mathbb{P}^{n}}(p)\longrightarrow C^{\bullet}(k)^{\oplus m}\longrightarrow C^{\bullet}(-1)\otimes\Omega^{p-1}_{\mathbb{P}^{n}}(p)\longrightarrow0.
$$
where $q=0,\dots,n,$ $p=1,\dots,n,$ and $m=\begin{pmatrix}n+1 \\ p\end{pmatrix}$. It is easy to see that
\begin{align}
&\mathbb{H}^{0}(C^{\bullet}(k)\otimes\Omega^{p}_{\mathbb{P}^{n}}(p))=0 \textnormal{ for all }p\textnormal{ and }k\leq-1;\notag \\
&\mathbb{H}^{q}(C^{\bullet}(-1)\otimes\Omega^{n}_{\mathbb{P}^{n}}(n))= \mathbb{H}^{q}(C^{\bullet}(-2))=0 \textnormal{ for all }q;\notag \\
&\mathbb{H}^{q}(C^{\bullet}(-1))=0\textnormal{ for all }q\neq1;\notag \\
&\mathbb{H}^{n}(C^{\bullet}(k)\otimes\Omega^{p}_{\mathbb{P}^{n}}(p))=0 \textnormal{ for all }p\textnormal{ and }k\geq-n.\notag
\end{align}
Setting $q=n-1,$ we also obtain
$$
\mathbb{H}^{n-1}(C^{\bullet}(k)\otimes\Omega^{p}_{\mathbb{P}^{n}}(p))=0 \textnormal{ for }p\leq n-1\textnormal{ and }k\geq-n-1,
$$
and so on. The final step to get the claim is a matter of applying the vanishing of the hypercohomologies to Theorem \ref{beilinson-perverse} above.

\bigskip

Now suppose $C^{\bullet}$ is an instanton perverse sheaf of trivial splitting type on a line $\ell$ in $\pn.$ By definition, $C^{\bullet}$ is quasi-isomorphic to a complex of the form
\begin{equation}
\label{equinp}
\oh_{\pn}(-1)^{\oplus c} \stackrel{\alpha}{\longrightarrow}
\oh_{\pn}^{\oplus a} \stackrel{\beta}{\longrightarrow} \oh_{\pn}(1)^{\oplus c}.
\end{equation}
Then we have
\begin{align*}
\calh^{0}(C^{\bullet}) &\simeq \ker\beta/{\rm im}\,\alpha\\
\calh^{1}(C^{\bullet}) &\simeq {\rm coker}\,\beta.
\end{align*}
We can break down (\ref{equinp}) into three short exact sequences
\begin{equation}\label{um}
0 \longrightarrow \ker\beta\longrightarrow \opn^{\oplus a} \stackrel{\beta}{\longrightarrow} {\rm im}\,\beta \longrightarrow 0
\end{equation}
\begin{equation}\label{dois}
0 \longrightarrow {\rm im}\,\beta \longrightarrow \opn(1)^{\oplus c} \longrightarrow \calh^{1}(C^{\bullet}) \longrightarrow 0
\end{equation}
\begin{equation}\label{tres}
0 \longrightarrow  \opn(-1)^{\oplus c} \stackrel{\alpha}{\longrightarrow} \ker\beta \longrightarrow \calh^{0}(C^{\bullet}) \longrightarrow 0.
\end{equation}
Twisting (\ref{um}), (\ref{dois}) and (\ref{tres}) by $\opn(k)$, with the right $k$, we get
\begin{align*}
{\rm for}\ n\ge2, &\ \ \ho^0(\calh^{0}(C^{\bullet})(-1))\simeq\ho^0(\ker\beta(-1))=0\\
{\rm for}\ n\ge3, &\ \ \ho^{n}(\ker\beta(-n))\simeq\ho^{n-1}({\rm im}\,\beta(-n))\simeq\ho^{n-2}(\calh^{1}(C^{\bullet})(-n))\\
{\rm for}\ n\ge3, &\ \ \ho^1(\calh^{0}(C^{\bullet})(-2))\simeq\ho^1(\ker\beta(-2))\simeq\ho^0({\rm im}\,\beta(-2))=0\\
{\rm for}\ n\ge4, &\ \ \ho^{n-1}(\calh^{0}(C^{\bullet})(1-n))\simeq\ho^{n-1}(\ker\beta(1-n))\simeq\ho^{n-2}({\rm im}\,\beta(1-n))\\
                  &\ \ \ \ \ \ \ \ \ \ \ \ \ \ \ \ \ \ \ \ \ \ \ \ \ \ \ \ \ \ \ \simeq\ho^{n-3}(\calh^{1}(C^{\bullet})(1-n))\\
{\rm for}\ n\ge4, &\ \ \ho^{p}(\calh^{0}(C^{\bullet})(k))\simeq\ho^{p}(\ker\beta(k))\simeq\ho^{p-1}({\rm im}\,\beta(k)),\ \forall k,\ {\rm for}\ 2\leq p\leq n-2\\
\end{align*}

Now choose a flag of linear subspaces of $\pn$ as in the following
$$\ell\subset L_{2}\subset\cdots L_{n-1}\subset\pn,$$ where the index of each space is its dimension.
By successive restrictions $$0\longrightarrow C^{\bullet}|_{L_{i}}(k-1)\longrightarrow C^{\bullet}|_{L_{i}}(k)\longrightarrow C^{\bullet}|_{L_{i-1}}(k)\longrightarrow0,$$ one reaches the complex $C^{\bullet}|_{\ell}(k),$ on the line $l,$ which is quasi-isomorphic to $\mathcal{O}_{\ell}^{\oplus r}(k)$ by the trivial splitting type hypothesis. Applying the same reasoning in the proof of \cite[Lemma 2.4]{N2}, one obtains $\ho^n(\mathcal{H}^0(C^\bullet)(-n))=0.$ Now the rest of the proof is a matter of using the vanishing properties above to the spectral sequence defining hypercohomology:

Let $(C^{\bullet},d)$ be a complex of sheaves and $(C^{\bullet\bullet},D)$ the Cartan-Eilemberg complex obtained by taking the \u Cech resolution of every term. Let $D=d+\delta$ be the total differential, where $\delta$ is the \u Cech differential. The defining spectral sequences for the hypercohomology of a complex $C^{\bullet}$, i.e.,
\begin{align*}
E^{p,q}_{2}&=\ho^{p}_{d}(\ho^{q}_{\delta}(C^{\bullet\bullet})) \\
'E^{p,q}_{2}&=\ho^{p}_{\delta}(\ho^{q}_{d}(C^{\bullet\bullet}))
\end{align*}
converge to $E_{\infty}^{p,q}=\mathbb{H}^{p+q}(C^{\bullet})$. For the complex $(C^{\bullet}(k),d)$, one can easily compute the $E_{2}-$term of the spectral sequence $'E^{p,q},$ given by
$$\xymatrix@R-1pc@C-1pc{
0&0&0&\dots&0&0 \\
\ho^{0}_{\delta}(\mathcal{Q})\ar[drr]^{d_{2}}& \ho^{1}_{\delta}(\mathcal{Q})\ar[drr]^{d_{2}}& \ho^{2}_{\delta}(\mathcal{Q})&\dots\ar[drr]^{d_{2}}&0&0  \\
\ho^{0}_{\delta}(\mathcal{F})&\ho^{1}_{\delta}(\mathcal{F})& \ho^{2}_{\delta}(\mathcal{F})&\dots&\ho^{n-1}_{\delta}(\mathcal{F})& \ho^{n}_{\delta}(\mathcal{F}) \\
0&0&0&\dots&0&0
}$$
where we put $\ho^{0}_{d}(C^{\bullet\bullet}(k)):=\mathcal{F}$ and $\ho^{1}_{d}(C^{\bullet\bullet}(k)):=\mathcal{Q}$. The differential of this term is $d_{2}:$ $'E^{p,q}_{2}$ $\to$ $'E^{p+2,q-1}_{2}$, vanishing everywhere but for $d_{2}:$ $'E^{p,1}_{2}$ $\to$ $'E^{p+2,0}_{2}$, as in the diagram above. Since the cohomology $H_{d_{r}}(E_{r})$ is canonically isomorphic to $E_{r+1}$ then the third term of the spectral $E_{3}$ is of the following form:

\begin{center}
\begin{tabular}{|c|c|c|c|c|c|}
\hline
0&0&$\dots$&0&0 \\ \hline
$E^{0,0}_{3}$&$E^{1,0}_{3}$&$\dots$&$E^{n-1,0}_{3}$& $E^{n,0}_{3}$ \\ \hline
0&0&$\dots$&0&0 \\ \hline
\end{tabular}
\end{center}
where the terms are given by $$E^{p,0}_{3}=\left\{\begin{array}{ll} \ho^{0}_{\delta}(\mathcal{F}) & p=0\\ \ho^{1}_{\delta}(\mathcal{F}) & p=1 \\ \frac{\ho^{p}_{\delta}(\mathcal{F})}{\ho^{p-2}_{\delta}(\mathcal{Q})}& p\geq2 \end{array}\right.$$ Since the differential $d_{3}$ is identically zero, then the spectral sequence degenerate. It follows that
$$\mathbb{H}^{p}(C^{\bullet})=\left\{\begin{array}{ll} \ho^{0}_{\delta}(\mathcal{F}) & p=0\\  \\\ho^{1}_{\delta}(\mathcal{F}) & p=1 \\ \\ \frac{\ho^{p}_{\delta}(\mathcal{F})}{\ho^{p-2}_{\delta}(\mathcal{Q})}& p\geq2 \end{array}\right.$$
and we are done.
\end{proof}

Following \cite[Def. 5.1]{HL}, we say that a perverse instanton sheaf on $\yy$ is stable if it comes from a stable ADHM data over $\yy$. In this sense, every instanton sheaf is stable as a perverse instanton sheaf. It also follows that the GIT quotient $\mathcal{M}_{\yy}^{\rm st}$ is a (fine) moduli space of stable framed perverse instanton sheaves on $\yy$. However, it would be interesting to have an intrinsic definition of stability in the derived category for perverse instanton sheaves that does not refer to the underlying ADHM data.



\begin{thebibliography}{99}

\bibitem{AO}
Ancona, V., Ottaviani, G.:
Stability of special instanton bundles on $\mathbb{P}^{2n+1}$.
Trans. Am. Math. Soc. {\bf 341} (1994), 677-693.

\bibitem{ADHM}
Atiyah, M., Drinfeld, V., Hitchin, N., Manin, Yu.:
Construction of instantons.
Phys. Lett. {\bf 65A} (1978), 185-187.

\bibitem{BH}
Barth, W., Hulek, K.:
Monads and moduli of vector bundles.
Manuscripta Math. {\bf 25} (1978), 323-347.

\bibitem{BZN}
Ben-Zvi, D., Nevins, T.:
Perverse bundles and Calogero-Moser spaces,
Compos. Math. {\bf 144} (2008), 1403-1428.

\bibitem{BFG}
Braverman, A., Finkelberg, M., Gaitsgory, D.:
Uhlenbeck Spaces via Affine Lie Algebras,
The Unity of Mathematics, Prog. in Math. {\bf 244} (2006), 17-135.

\bibitem{CTT}
Coand\u{a}, I., Tikhomirov, A., Trautmann, G.:
Irreducibility and smoothness of the moduli space
of mathematical 5-instantons over $\p3$.
Int. J. Math. {\bf 14} (2003), 1-45.

\bibitem{CM}
Costa, L., Mir\'o-Roig, R. M.:
Monads and Instanton Bundles on Smooth Hyperquadrics.
Math. Nachr. {\bf 282} (2009), 169-179.

\bibitem{CO}
Costa, L., Ottaviani, G.:
Nondegenerate multidimensional matrices and instanton bundles.
Trans. Am. Math. Soc. {\bf 355} (2002), 49-55.

\bibitem{D}
Diaconescu, D.-E.:
Moduli of ADHM sheaves and local Donaldson-Thomas theory.
Preprint math/0801.0820.

\bibitem{D1}
Donaldson, S.:
Instantons and Geometric Invariant Theory.
Commun. Math. Phys. {\bf 93} (1984), 453-460.

\bibitem{F}
Floystad, G.:
Monads on projective spaces.
Comm. Algebra {\bf 28} (2000), 5503-5516.

\bibitem{FJ2}
Frenkel, I. B., Jardim, M.:
Complex ADHM equations, and sheaves on $\p3$.
J. Algebra {\bf 319} (2008), 2913-2937.

\bibitem{GM}
Gelfand, Yu. Manin,
Methods of homological algebra.

\bibitem{Godement} Godement, R.:
Topologie alg\'ebrique et th\'eorie des faisceaux,
Hermann, Paris (1958).

\bibitem{GK}
Gothen, P. B., King, A. D.:
Homological algebra of twisted quiver bundles.
J. London. Math. Soc. {\bf 71} (2005), 85-99.

\bibitem{grif}
Grifiths, P., Harris, J.:
Principles of algebraic geometry,
New York, Wiley-Interscience, 1979.

\bibitem{Grothendieck0}
Grothendieck, A.:
Sur quelques points d'alg�bre homologique,
T\^ohoku. Math. Journ. I, t. IX (1956), 119-221.

\bibitem{Grothendieck}
Grothendieck, A.:
El\'ements de g\'eom\'etrie alg\'ebrique III, 1,
Publ. Math. IHES, no 11 (1961), 5-165.

\bibitem{HRS}
Happel, D., Reiten, I., Smalo, S.:
Tilting in abelian categories and quasitilted algebras.
Mem. Amer. Math. Soc. {\bf 120} (1996).

\bibitem{HL}
Hauzer, M., Langer, A.:
{\em Moduli Spaces of Framed Perverse Instanton Sheaves on} $\mathbb{P}^3$.
Glasgow Math. J. {\bf 53} (2011), 51–96.

\bibitem{Hart}
Hartshorne, R.:
Algebraic geometry,
Springer-Verlag, GTM, 52.

\bibitem{Hart1}
Hartshorne, R.:
Residues and duality,
LNM 20, Springer, 1966.

\bibitem{henni}
Henni, A. A.:
Monads for torsion-free sheaves on multi-blow-ups of the projective plane.
Preprint arXiv:0903.3190.

\bibitem{Hilton}
Hilton, P. J. ,Stammbach, U.:
A course in Homological Algebra.
Second edition. Graduate Texts in Mathematics, 4. Springer-Verlag, New York, 1997.

\bibitem{J-i}
Jardim, M.:
Instanton sheaves on complex projective spaces.
Collec. Math. {\bf 57} (2006), 69-91.

\bibitem{J-cr}
Jardim, M.:
Atiyah--Drinfeld--Hitchin--Manin construction of framed instanton sheaves.
C. R. Acad. Sci. Paris, Ser. I {\bf 346} (2008), 427-430.

\bibitem{J-pn}
Jardim, M.:
Moduli spaces of framed instanton sheaves on projective spaces.
Preprint arXiv:0810.2550v2.

\bibitem{JVM}
Jardim, M., Martins, R. V.:
Linear and Steiner bundles on projective varieties.
Comm. Alg. {\bf 38} (2010), 2249-2270.

\bibitem{JVM2}
Jardim, M., Martins, R. V.:
The ADHM Variety and Perverse Coherent Sheaves.
J. Geom. Phys. {\bf 61} (2011), 2219-2232.

\bibitem{Ka}
Kashiwara, M.:
$t$-structures on the derived categories of holonomic $\mathcal D$-modules.
and coherent $\mathcal O$-modules.
Moscow Math. J. {\bf 4} (2004), 847-868.

\bibitem{Le}
Le Potier, J.:
Fibr\'e stable de rang 2 sur $\mathbb{P}^{2}(\mathbb{C}).$
Math. Ann. {\bf 241} (1979), 217-256.

\bibitem{MCS}
Mamone Capria, M., Salamon, S. M.:
Yang-Mills fields on quaternionic spaces.
Nonlinearity {\bf 1} (1988), 517-530.

\bibitem{Mar}
Maruyama, M.:
Instantons anxd parabolic sheaves.
Geometry and analysis (Bombay, 1992), 245�267, Tata Inst. Fund. Res., Bombay, 1995.

\bibitem{mumford}
Mumford, D.:
Geometric invariant theory.
Ergebnisse der Mathematik und ihrer Grenzgebiete, Neue Folge, Band 34 Springer-Verlag, Berlin-New York 1965.

\bibitem{N2}
Nakajima, H.:
{\em Lectures on Hilbert schemes of points on surfaces}.
Providence: American Mathematical Society, 1999

\bibitem{Ne}
Newstead, P.:
{\em Lectures on introduction to moduli problems and orbit spaces}.
Berlin: Springer-Verlag (1978).

\bibitem{OS}
Okonek, C., Spindler, H.:
Mathematical instanton bundles on $\mathbb{P}^{2n+1}$.
J. Reine Agnew. Math. {\bf 364} (1986), 35-50.

\bibitem{OSS}
Okonek, C., Schneider, M., Spindler, H.:
Vector bundles on complex projective spaces.
Progress in mathematics 3, Birkhauser, Boston, 1980.

\bibitem{ST}
Spindler, H., Trautmann, G.:
Special instanton bundles on $P_{2N+1}$, their geometry and their moduli,
Math. Ann. {\bf 286} (1990), 559–592.


\end{thebibliography}
\end{document}